\documentclass[11pt,a4paper,reqno]{amsart}%
\usepackage{a4wide}
\usepackage[table,usenames, dvipsnames]{xcolor}
\usepackage{amsfonts,amsmath,amssymb,amsthm,tabularx,ifthen,twoopt,enumerate}
\usepackage{graphicx}
\usepackage[latin1]{inputenc}
\usepackage[figuresright]{rotating}
\usepackage{natbib}
\usepackage{tikz}

\newcounter{hypA}
\newenvironment{hypA}{\refstepcounter{hypA}\begin{itemize}
  \item[({\bf A\arabic{hypA}})]}{\end{itemize}}

\def\P{\mathbb{P}}

\newcommand{\gp}[1]{\left(#1\right)}
\newcommand{\ga}[1]{\left\{#1\right\}}

\newcommand{\probT}[1]{\P\gp{#1}}

\def \1{\mathbf{1}}
\def\B{\textrm{B}}

\def \Kstar{K^\star}

\newtheorem{theo}{Theorem}

\newtheorem{lemma}{Lemma}

\theoremstyle{remark}
\newtheorem{remark}{Remark}

\newcommand{\corr}[1]{\textcolor{black}{#1}}

\title[Estimating the number of block boundaries without penalization]
{Estimating the number of block boundaries from diagonal blockwise matrices without penalization}

\author{V. Brault, M. Delattre, E. Lebarbier, T. Mary-Huard, C. L\'evy-Leduc}
\address{AgroParisTech/UMR INRA MIA 518}
\email{vincent.brault@agroparistech.fr}

\theoremstyle{plain}

\theoremstyle{definition}

\theoremstyle{remark}

\newcommand{\Y}{Y}
\newcommand{\K}{K}
\newcommand{\ttt}{t}
\newcommand{\D}{D}
\newcommand{\kk}{k}
\newcommand{\n}{n}
\newcommand{\ii}{i}
\newcommand{\jj}{j}
\newcommand{\E}{E}
\newcommand{\R}{R}
\newcommand{\el}{\ell}

\newcommand{\tauu}{\tau}

\DeclareMathOperator*{\argmin}{\operatorname{Argmin}}

\newcommand{\Yij}{\Y_{\ii,\jj}}
\newcommand{\vrai}{\star}
\newcommand{\Ks}{\K^\vrai}
\newcommand{\bts}{\mathbf{\ttt}^\vrai}
\newcommand{\ts}{\ttt^\vrai}
\newcommand{\tsKs}{\ts_{\Ks}}
\newcommand{\Ds}{\D^\vrai}
\newcommand{\Dsk}{\Ds_{\kk}}
\newcommand{\tsk}{\ts_{\kk}}
\newcommand{\muu}{\mu}

\newcommand{\Ez}{\E_0}
\newcommand{\Ezs}{\E_0^\vrai}

\newcommand{\muks}{\muu_{\kk}^\vrai}
\newcommand{\mus}{\muu^\vrai}
\newcommand{\muzs}{\muu_{0}^\vrai}
\newcommand{\bt}{\mathbf{\ttt}}

\newcommand{\bY}{\mathbf{\Y}}
\newcommand{\Dk}{\D_{\kk}}

\newcommand{\muz}{\muu_{0}}

\newcommand{\tK}{\ttt_{\K}}
\newcommand{\tk}{\ttt_{\kk}}

\newcommand{\Q}{Q}
\newcommand{\J}{J}
\newcommand{\V}{V}
\newcommand{\W}{W}

\newcommand{\G}{G}
\newcommand{\Z}{Z}
\newcommand{\A}{A}
\newcommand{\e}{e}
\newcommand{\eps}{\varepsilon}
\newcommand{\Deltaa}{\Delta}
\newcommand{\lambdaa}{\lambda}
\newcommand{\deltaa}{\delta}
\newcommand{\betaa}{\beta}
\newcommand{\gammaa}{\gamma}
\newcommand{\epsij}{\eps_{\ii,\jj}}
\newcommand{\Kmax}{\K_{\max}}
\newcommand{\Qn}{\Q_{\n}}
\newcommand{\YbarDk}{\overline{\Y}_{\!\!\Dk}}

\newcommand{\Jn}{\J_{\n}}
\newcommand{\mbE}{\mathbb{E}}
\newcommand{\Kn}{\B_{\n}}
\newcommand{\Vn}{\V_{\n}}
\newcommand{\Wn}{\W_{\n}}

\newcommand{\Gzero}{\G_{00}}
\newcommand{\Gun}{\G_{01}}
\newcommand{\YbarGun}{\overline{\Y}_{\!\!\Gun}}

\newcommand{\YbarDsk}{\overline{\Y}_{\!\!\Dsk}}
\newcommand{\Jzn}{\Jn^{0}}

\newcommand{\KDn}{\Kn^{\D}}
\newcommand{\VDn}{\Vn^{\D}}
\newcommand{\WDn}{\Wn^{\D}}
\newcommand{\Gszero}{\Gzero^{\vrai}}
\newcommand{\ip}{{\ii'}}
\newcommand{\jp}{{\jj'}}
\newcommand{\epsipjp}{\eps_{\ip,\jp}}
\newcommand{\Zn}{\Z_{\n}}
\newcommand{\Kzn}{\Kn^{0}}
\newcommand{\Vzn}{\Vn^{0}}
\newcommand{\Wzn}{\Wn^{0}}
\newcommand{\nkl}{\n_{\kk,\el}}
\newcommand{\Dsl}{\Ds_{\el}}
\newcommand{\nzl}{\n_{0,\el}}

\newcommand{\nk}{\n_{\kk}}
\newcommand{\nsl}{\n_{\el}^{\vrai}}
\newcommand{\nz}{\n^{0}}
\newcommand{\mcA}{\mathcal{A}}
\newcommand{\AnK}{\mcA_{\n,\K}}
\newcommand{\Deltan}{\Deltaa_{\n}}
\newcommand{\AnKDelta}{\AnK^{\Deltan}}
\newcommand{\AnKsDelta}{\mcA_{\n,\Ks}^{\Deltan}}

\newcommand{\muls}{\muu_{\el}^\vrai}
\newcommand{\lp}{{\el'}}
\newcommand{\mulps}{\muu_{\lp}^\vrai}
\newcommand{\nklp}{\n_{\kk,\lp}}
\newcommand{\lun}{{\el_{1}}}
\newcommand{\muluns}{\muu_{\lun}^\vrai}
\newcommand{\nklun}{\n_{\kk,\lun}}
\newcommand{\ldeux}{{\el_{2}}}
\newcommand{\muldeuxs}{\muu_{\ldeux}^\vrai}
\newcommand{\nkldeux}{\n_{\kk,\ldeux}}

\newcommand{\btau}{\boldsymbol{\tauu}}
\newcommand{\btaus}{\btau^{\vrai}}
\newcommand{\Deltataus}{\Deltaa_{\btau}^{\vrai}}
\newcommand{\taus}{\tauu^{\vrai}}
\newcommand{\tausk}{\taus_{\kk}}

\newcommand{\lambdainf}{\underline{\lambdaa}}
\newcommand{\tl}{\ttt_{\el}}
\newcommand{\thatl}{\that_{\el}}
\newcommand{\thatk}{\that_{\kk}}
\newcommand{\tsl}{\ts_{\el}}

\newcommand{\tp}{{\ttt'}}
\newcommand{\mcI}{\mathcal{I}}

\newcommand{\lambdasup}{\overline{\lambdaa}}

\newcommand{\Prob}{\mathbb{P}}

\newcommand{\lambdainfz}{{\lambdainf^{(0)}}}

\newcommand{\elpk}{\e_{\lp,\kk}}

\newcommand{\nl}{\n_{\el}}
\newcommand{\bthat}{\widehat{\bt}}
\newcommand{\that}{\widehat{\ttt}}
\newcommand{\bthatK}{\bthat_{\K}}
\newcommand{\bthatKhat}{\bthat_{\Khat}}

\newcommand{\Khat}{\widehat{\K}}

\newcommand{\mcID}{\mcI_{\D}}
\newcommand{\kp}{\kk'}
\newcommand{\Dkp}{\D_{\kp}}

\newcommand{\mcIRK}{\mcI_{\R}^{\Kmax}}

\newcommand{\Dl}{\D_{\el}}

\newcommand{\ensemblelemme}{\left\{\bt\in\AnK^{1/\n},
\left\|\bt-\bts\right\|_{\infty} > \n\delta \right\}}
\newcommand{\ensemblelemmestar}{\left\{\bt\in\mcA_{\n,\Ks}^{1/\n}, \left\|\bt-\bts\right\|_{\infty} > \n\delta \right\}}

\newcommand{\compl}[1]{{#1}^C}

\usepackage{hyperref}

\date{\today}

\begin{document}

\maketitle

\begin{abstract}
In computational biology, numerous recent studies have been dedicated to the analysis of the chromatin structure 
within the cell by two-dimensional segmentation methods. Motivated by this application, we consider the problem of retrieving 
the diagonal blocks in a matrix of observations.
\textcolor{black}{The theoretical properties of the least-squares estimators of both the boundaries and
the number of blocks proposed by \cite{levy2014two} are investigated. More precisely, the contribution of the paper is
to establish the consistency of these estimators. \corr{A surprising
  consequence of our results is that, contrary to the one-dimensional
  case, a penalty is not needed} for retrieving the true number of diagonal blocks.} Finally, the results are illustrated on synthetic data.


\end{abstract}


\section{Introduction}

Detecting change-points in one-dimensional signals is a very important task which arises in
many applications, ranging from EEG (Electroencephalography) to speech processing and network intrusion detection,
see \cite{basseville:1993,brodsky:2000,tartakovsky:nikiforov:basseville:2014}. The aim of such approaches
is to split a signal into several homogeneous segments according to some quantity.
A large literature has been dedicated to the change-point detection issue for one-dimensional data.
This problem may also have several applications when dealing with two-dimensional data.
One of the main situations in
which this problem occurs is the detection of chromosomal regions having close spatial location in
the nucleus of a cell. Detecting such regions provides valuable insight to understand the influence of chromosomal
conformation on cell functioning.
More precisely, we will consider the problem of identifying the so-called \emph{cis}-interactions between regions of a chromosome. In this context, $n$ locations spatially ordered along a given chromosome are considered, the goal being to find clusters of adjacent locations that strongly interact. The elements $\Yij$ of a data matrix $\mathcal{Y}$ will then correspond to the interaction level between locations $i$ and $j$ of a chromosome, which can be measured using the recently developed HiC technologies, see \cite{dixon2012topological}. In this application, the signal - and consequently the data matrix - exhibits a strong structure: one should observe high signal levels within blocks of locations along the matrix diagonal, and a signal that is close to some (low) baseline level everywhere else.

As shown in \cite{levy2014two}, the identification of \emph{cis}-interactions can be cast as a segmentation problem, where
the goal is to identify diagonal blocks (or regions) with homogeneous interaction levels.
Thanks to the spatial repartition of these regions along the diagonal, the two-dimensional segmentation of the data matrix
actually boils down to a particular one-dimensional segmentation. The dynamic programming algorithm originally proposed by
\cite{bellman1961approximation} is well-known to provide the exact solution of the one-dimensional segmentation issue in the least-squares sense. Therefore we benefit from the data structure by avoiding both the computational burden and the approximation errors that come with heuristic methods used to solve the complex generic problem of two-dimensional segmentation.

While being able to handle large interaction data matrices from an algorithmic point of view, model selection (\textit{i.e.} selecting the number of blocks $K$) remains an open question when dealing with such data. This is contrasted with the problem of one-dimensional signal segmentation, for which the properties of the estimators have been largely addressed for instance in \cite{boysen2009,lavielle2000least,yao:au:1989}.
In these approaches, the number of change-points is usually performed thanks to a Schwarz-like penalty $\lambda_n K$
where $\lambda_n$ is often calibrated on data, as in \cite{lavielle:2005} and \cite{lavielle2000least},
or a penalty $K(a + b \log(n/K))$ as in \cite{lebarbier:2005} and \cite{massart:2004}, where $a$ and $b$ are data-driven as well.

The goal of the present paper is to \textcolor{black}{prove the consistency} of the estimators of both the boundaries and
the number of blocks obtained by minimizing \textcolor{black}{the (slightly modified)} least-squares criterion proposed by \cite{levy2014two}.
The proof relies on the strong structure of the data which is of great help for the model selection issue and for the algorithmic aspects.
More precisely, we will prove that the \emph{non-penalized} least-squares estimators of the number of blocks is consistent.

The paper is organized as follows: Section~\ref{Section: StatFram} introduces the modeling of the data and the definition of the least-squares estimators that will be considered throughout the article. The theoretical properties of the estimators are derived in Section~\ref{Section: TheoResults} and illustrated on synthetic data in Section~\ref{Section: NumExp}. \textcolor{black}{A discussion is given in Section~\ref{Section:Discussion}. The technical aspects of the proofs are detailed in Section~\ref{Section: appendix} and in the supplementary material.}



\section{Statistical Framework}\label{Section: StatFram} 

\subsection{Modeling}

Let us consider $\mathcal{Y}=(Y_{i, j})_{1\leq i,j\leq n}$, a symmetric matrix of random variables.
Because of the symmetry, we shall focus on its upper-triangular part denoted by $\bY=(Y_{i, j})_{1 \leq i \leq j \leq n}$
where the $Y_{i,j}$ will be assumed to be independent and such that
\begin{equation}\label{eq:model}
\Yij=\mathbb{E}\left[\Yij\right]+\epsij=\mu_{i, j}+\epsij,\; 1 \leq i \leq j \leq n.
\end{equation}
The $\epsij$ satisfy the following assumption:
\begin{hypA}\label{hyp:eps}
The $\epsij$ are assumed to be centered, i.i.d. and such that there exists a positive constant $\betaa$
such that for all $\nu\in\mathbb{R}$,
\begin{equation*}
\mathbb{E}\left[e^{\nu\varepsilon_{11}}\right]\leq e^{\betaa\nu^2}.
\end{equation*}
\end{hypA}
We shall moreover assume that the matrix of means $(\mu_{i,j})_{1\leq i\leq j\leq n}$ is block diagonal.
\textcolor{black}{More precisely, let $\btaus = (\taus_0,\taus_1,\dots,\taus_{\Kstar})$ be a vector
of break fractions such that $0=\taus_0<\taus_1<\dots<\taus_{\Kstar}=1$. In what follows, the break fractions are fixed 
quantities: neither their number nor their positions change when $n$ grows.}
The parameters
$\mu_{i, j}$ are such that
\begin{align}\label{eq:bloc_mu}
\mu_{i, j} & =  \mu_k^\star \quad \text{if } (i, j) \in D_k^\star, \; k = 1, \dots, K{^\star}, \nonumber\\
& =  \mu_0^\star \quad \text{if } (i, j) \in E_0^\star,
\end{align}
where the (half) diagonal blocks $D_k^\star$ ($k = 1, \dots, K{^\star}$) are defined as follows,
\begin{equation}\label{eq:Dkstar}
D_k^\star = \{(i,j): t_{k-1}^\star\leq i\leq j\leq t_{k}^\star-1\},
\end{equation}
where \textcolor{black}{$t_k^\star=[n\btaus_k]+1$ are thus such that
$1=t_0^\star<t_1^\star<\dots<t_{\Kstar}^\star=n+1$, $[x]$ denoting the integer part of $x$.} They
stand for the true block boundaries and $\Kstar$ corresponds to the true number of blocks.
In Equation~(\ref{eq:bloc_mu}), $\Ezs$ corresponds to the set of positions lying outside the diagonal blocks:
\begin{equation}\label{eq:E0}
E_0^\star = \{(i,j): 1 \leq i \leq j \leq n\} \cap \compl{\left(\cup D_k^\star  \right)},
\end{equation}
where $\compl{A}$ denotes the complement of set $A$.
An example of such a matrix is displayed in Figure \ref{fig:exemple} (left).
 The following will also be assumed for the true block sizes:
\begin{hypA}\label{Hypothese:TailleMaxBlocks}
For all $\ell$, one has
\begin{eqnarray*}
0<\Deltataus=\underset{\kk\in\{1,\ldots,\K^{\vrai}\}}{\min}\;\;\left|\tausk-\taus_{\kk-1}\right| \leq |\taus_{\ell+1}-\taus_\ell| \leq c,
\end{eqnarray*}
where $c \in (0,1)$ is a known constant. 
\end{hypA}

Moreover the $\muks$ satisfy the following assumption:
\begin{hypA}\label{Hypothese:mon_rupt}
\hspace{4cm}$\displaystyle{\lambdainfz=\underset{1\leq\kk\leq\Ks}{\min}\left|\muks-\muzs\right|>0}$.
\end{hypA}


\subsection{Inference}
In this framework, the inference consists in estimating both the number of blocks and the true break fraction vector $\btaus$ (or equivalently the true boundary vector $\bts$). One strategy would be to use the following least-squares criterion:
\begin{eqnarray}\label{eq:LS_criterion}
\bthatK^{\textrm{LS}} \in \argmin_{\bt\in\AnKDelta} \ \ga{ \left[\sum_{\kk=1}^{\K}{\sum_{(\ii,\jj)\in\Dk}{\left(\Yij-\YbarDk\right)^2}}\right]+\sum_{(\ii,\jj)\in\Ez}{\left(\Yij-\overline{\Y}_{\!\!\Ez}\right)^2} }, 
\end{eqnarray}
where $\overline{\Y}_{\!\!\mathcal{D}}$ is the empirical mean of the $Y_{i,j}$ when the indices $(i,j)$ belong to $\mathcal{D}$, $\Dk$ and $\Ez$ are defined as in (\ref{eq:Dkstar}) and (\ref{eq:E0}) except that $\bts$ is replaced by $\bt$, and $K$ is the considered number of segments -- $\Kstar$
being unknown in practice. Moreover,
\begin{multline}\label{eq:AnK}
\AnKDelta=\left\{\bt=(\ttt_0,\ldots,\ttt_{\K}):\ttt_0=1<\ttt_1<\ldots<\ttt_{K}=\n+1\right.\\
          \left.\quad\text{ and }\forall 1\leq\kk\leq\K,\;\n\Deltan\leq\tk-\ttt_{\kk-1}< c\n\right\}
\end{multline}
is the set of admissible segmentations, where $\Delta_n$ denotes a positive sequence.

%
However, thanks to~(A\ref{Hypothese:TailleMaxBlocks}), one can derive an unbiased estimator of $\muzs$ using the upper-right triangle part of the matrix $\mathcal{Y}$ denoted $\Gun$ and defined by
\begin{equation}\label{eq:G01}
\Gun = \{(i, j): 1\leq i\leq \nz,(n-\nz+1)\leq j\leq n\}\quad \textrm{with } \nz = \left[(1-c)n\right].
\end{equation}
Indeed the intersection between the blocks $\Dk$ and $\Gun$ will always be empty. Thus, we can split $E_0^\star$ into two disjoint sets $\Gszero$ and $\Gun$ (see the right part of Figure~\ref{fig:exemple}) as follows,
\begin{equation}\label{eq:G00:G01}
E_0^\star=\Gszero\cup\Gun.
\end{equation}
Consequently, we will consider the following slightly modified least-squares criterion:
\begin{equation}\label{def:Qn}
\bthatK \in \argmin_{\bt\in\AnKDelta} \ \Qn^{\K}(\bt),
\end{equation}
where
\begin{equation}\label{def:QnK}
\Qn^{\K}(\bt)=\ga{\left[\sum_{\kk=1}^{\K}{\sum_{(\ii,\jj)\in\Dk}{\left(\Yij-\YbarDk\right)^2}}\right] +\sum_{(\ii,\jj)\in\Ez}{\left(\Yij-\YbarGun\right)^2}}.
\end{equation}
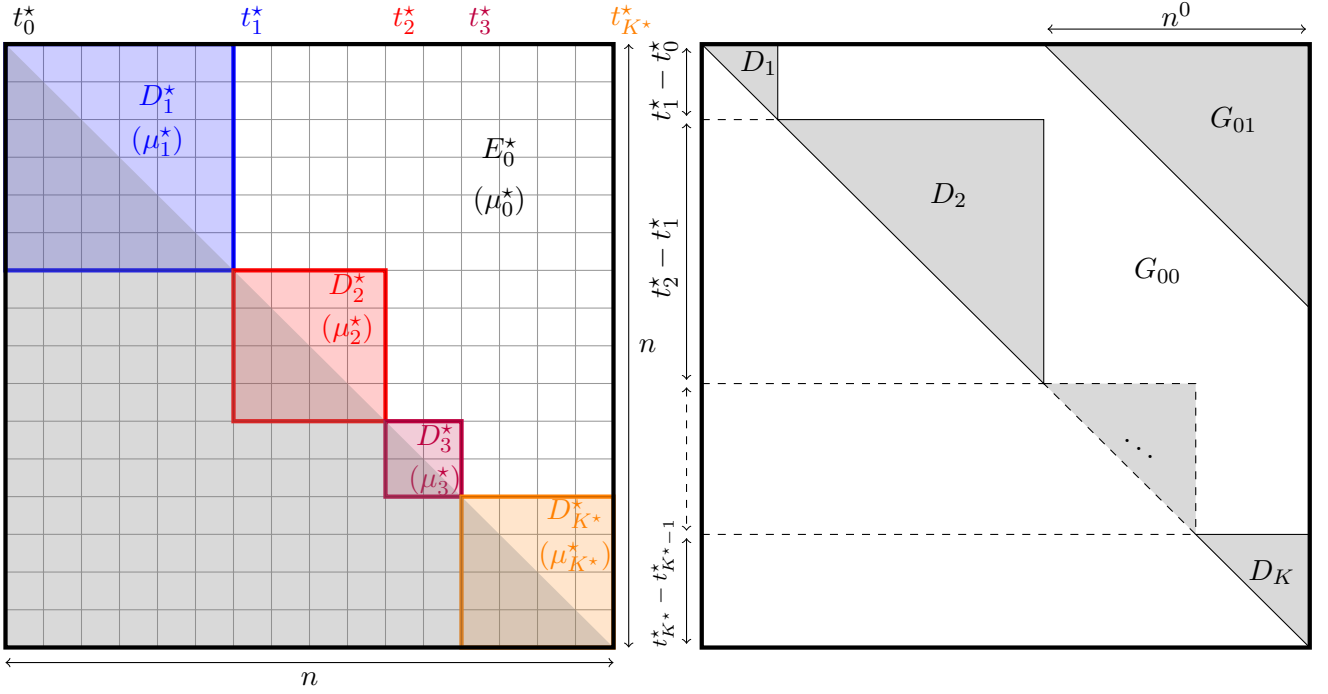
\begin{figure}[!h]
\begin{center}
\begin{tabular}{cc}
\hspace{-9mm}
\begin{tikzpicture}
\draw[opacity=0.4] (0,0) grid[step=0.5] (8,8);
\draw[<->] (0,-0.2) -- (8,-0.2);
\draw (4,-0.2) node[below]{$n$};
\draw[<->] (8.2,0) -- (8.2,8);
\draw (8.2,4) node[right]{$n$};
\draw[ultra thick,color=blue] (0,8) rectangle (3,5);
\draw[fill=blue,opacity=0.2] (0,8) rectangle (3,5);
\draw (3.25,8) node[above,color=blue]{$\ts_1$};
\draw (2,6.95) node[above,color=blue]{$\Ds_1$};
\draw (2,7.05) node[below,color=blue]{$(\mus_1)$};
\draw[ultra thick,color=red] (3,5) rectangle (5,3);
\draw[fill=red,opacity=0.2] (3,5) rectangle (5,3);
\draw (5.25,8) node[above,color=red]{$\ts_2$};
\draw (4.5,4.45) node[above,color=red]{$\Ds_2$};
\draw (4.5,4.55) node[below,color=red]{$(\mus_2)$};
\draw[ultra thick,color=purple] (5,3) rectangle (6,2);
\draw[fill=purple,opacity=0.2] (5,3) rectangle (6,2);
\draw (6.25,8) node[above,color=purple]{$\ts_3$};
\draw (5.65,2.45) node[above,color=purple]{$\Ds_3$};
\draw (5.65,2.55) node[below,color=purple]{$(\mus_3)$};
\draw[ultra thick,color=orange] (6,2) rectangle (8,0);
\draw[fill=orange,opacity=0.2] (6,2) rectangle (8,0);
\draw (8.25,8) node[above,color=orange]{$\tsKs$};
\draw (7.5,1.45) node[above,color=orange]{$\Ds_{\Ks}$};
\draw (7.5,1.55) node[below,color=orange]{$(\mus_{\Ks})$};

\draw (6.5,6.25) node[above]{$\Ezs$};
\draw (6.5,6.25) node[below]{$(\muzs)$};

\draw (0.25,8) node[above]{$\ts_0$};

\fill[color=gray,opacity=0.3] (0,0) -- (0,8) -- (8,0) -- cycle;
\draw[ultra thick] (0,0) rectangle (8,8);
\end{tikzpicture}
&
\hspace{-7mm}
\begin{tikzpicture}

\fill[color=gray,opacity=0.3] (0,8) -- (1,8) -- (1,7) -- cycle;
\draw (0,8) -- (1,8) -- (1,7) -- cycle;
\draw (0.75,7.75) node{$\D_{1}$};
\fill[color=gray,opacity=0.3] (1,7) -- (4.5,7) -- (4.5,3.5) -- cycle;
\draw (1,7) -- (4.5,7) -- (4.5,3.5) -- cycle;
\draw (3.25,6) node{$\D_{2}$};
\fill[color=gray,opacity=0.3] (4.5,3.5) -- (6.5,3.5) -- (6.5,1.5) -- cycle;
\draw[dashed] (4.5,3.5) -- (6.5,3.5) -- (6.5,1.5) -- cycle;
\draw (5.75,2.75) node{$\ddots$};
\fill[color=gray,opacity=0.3] (6.5,1.5) -- (8,1.5) -- (8,0) -- cycle;
\draw (6.5,1.5) -- (8,1.5) -- (8,0) -- cycle;
\draw (7.5,1) node{$\D_{\K}$};

\draw[<->] (-0.2,7.95) -- (-0.2,7.05);
\draw (-0.5,8.2) node[left,rotate=90]{$\ts_{1}-\ts_{0}$};
\draw[dashed] (0,7) -- (1,7);
\draw[<->] (-0.2,6.95) -- (-0.2,3.55);
\draw (-0.5,5.85) node[left,rotate=90]{$\ts_{2}-\ts_{1}$};
\draw[dashed] (0,3.5) -- (4.5,3.5);
\draw[<->,dashed] (-0.2,3.45) -- (-0.2,1.55);
\draw[dashed] (0,1.5) -- (6.5,1.5);
\draw[<->] (-0.2,1.45) -- (-0.2,0.05);
\draw (-0.5,1.85) node[left,rotate=90]{\scriptsize{$\ts_{\Ks}-\ts_{\Ks-1}$}};

\fill[color=gray,opacity=0.3] (4.5,8) -- (8,8) -- (8,4.5) -- cycle;
\draw (4.5,8) -- (8,8) -- (8,4.5) -- cycle;
\draw (7,7) node{$\Gun$};
\draw (6,5) node{$\Gzero$};
\draw[<->] (4.55,8.2) -- (7.95,8.2);
\draw (6.25,8.1) node[above]{$\nz$};

\draw[ultra thick] (0,0) rectangle (8,8);

\draw (1,-0.5) node{ $ $ };
\end{tikzpicture}
\\
\end{tabular}
\end{center}
\caption{\label{fig:exemple} Left: Example of a matrix $(\mu_{i,j})$ with $\n=16$ and $\Ks=4$. Right: Illustration of the notations used in the estimation criterion.}
\end{figure}
Lastly, we will consider the following estimator of $\Kstar$:
\begin{equation}\label{eq:def:Khat}
\Khat=\argmin_{1\leq K\leq \Kmax} \Qn^{\K}\left(\bthatK\right),
\end{equation}
where $\bthatK$ is defined in \eqref{def:Qn} and $\Kmax$ is the maximal number of blocks considered.

\corr{Criterion (\ref{eq:def:Khat}) based on (\ref{def:QnK}) has been proposed by \cite{levy2014two}. The goal of our paper
is to validate this latter approach theoretically. Note that the main
difference between (\ref{eq:LS_criterion}) and (\ref{def:QnK}) is the
estimation of $\muzs$ that is independent from the
segmentation, since $\Gun$ is fixed. Hence, $\muzs$ can be estimated prior to the
optimization of the criterion (\ref{def:QnK}). As a consequence, this
optimization can be performed by  using the dynamic programming
algorithm as explained in \cite{levy2014two}.
}

\section{Theoretical results}\label{Section: TheoResults}
The goal of this section is to derive the consistency of  $\Khat$ and $\boldsymbol{\hat{\tau}}$.
To prove these results, we shall need the following assumption on $\Deltan$:
\corr{
\begin{hypA}\label{Hypothese:ConsistanceNbRuptureMax1}
\hspace{2cm}$\displaystyle{\Deltan\frac{\sqrt{\n}}{(\log\n)^{1/4}}\underset{\n\rightarrow+\infty}{\longrightarrow}+\infty}$
\textrm{ and}\quad $\Deltan\leq \Deltataus$, \textrm{ for large enough } $n$.
\end{hypA}
}


%

\begin{theo}\label{theo:theo}
Let $\Yij$ be defined by (\ref{eq:model}). Assume that (A\ref{hyp:eps}), (A\ref{Hypothese:TailleMaxBlocks}), (A\ref{Hypothese:mon_rupt}) and (A\ref{Hypothese:ConsistanceNbRuptureMax1}) hold. Then $\Khat$
defined in (\ref{eq:def:Khat}) is such that:
\begin{equation}\label{eq:theo}
\P\left(\Khat\neq\Ks\right)\longrightarrow 0,\textrm{ as } n\to+\infty.
\end{equation}
\end{theo}

\begin{remark}\label{rem1}
Observe that, contrary to classical statistical frameworks, $\Khat$ is a consistent estimator of $\Ks$ even if it is
obtained without any penalization.
\end{remark}


\begin{remark}
In Theorem \ref{theo:theo}, the estimator $\Khat$ is defined as the minimizer of $Q_n^{K}(\bthatK)$ where $\bthatK$ is obtained by minimizing
$Q_n^{K}(\bt)$ over the set $\AnKDelta$. If we are only interested in proving that \textcolor{black}{$\P(\Khat<K^{\star})\to 0$}, the minimization can be performed  on the set  $\AnK^{1/n}$  instead of $\AnKDelta$, $\textit{i.e}$ without any constraint on the minimal distance between two consecutive change-points, see Lemma \ref{Lemme:MinorationKn} (i) below and Lemmas \ref{lem:maj:Vn}, \ref{lem:maj:Wn} and \ref{Lemme:InegaliteConcentrationZn}, which
are given in Section \ref{Section: appendix}.
\end{remark}

\begin{remark}
Theorem \ref{theo:theo} is valid under (A\ref{Hypothese:TailleMaxBlocks}) which implies that the number of observations
within each segment increases linearly with $n$, since $t_k^\star=[n\btaus_k]+1$. This assumption could be alleviated by assuming that $\Deltataus$ is no longer a constant. 
In that case, we shall need to assume that $\Deltataus n^{1/4}/(\log n)^{1/8}$ tends to infinity, as $n$ tends to infinity.
\end{remark}

\corr{
\begin{remark}
%
%
The assumption $\Deltan\gg (\log\n)^{1/4}/\sqrt{n}$ of
(A\ref{Hypothese:ConsistanceNbRuptureMax1}) can be understood in the
light of Lemma \ref{Lemme:MinorationKn} (ii) and Equation (\ref{eq:equation:fin:theo}) at the
end of the proof of Theorem \ref{theo:theo}. It is required to ensure
the convergence to zero of the exponential inequalities of the
random parts given in Lemmas \ref{lem:maj:Vn}, \ref{lem:maj:Wn} and \ref{Lemme:InegaliteConcentrationZn}.
%
%
This assumption 
is only required for proving that $\Prob\left(\Khat>\Ks\right)$ tends to zero as $n$ tends to infinity. As a consequence, 
when the number of blocks is known ($\Khat=K^\star$), the break fractions consistency is obtained in our paper when $\Delta_n=1/n$. 
Such a choice is impossible in the one-dimensional segmentation framework of \cite{lavielle2000least} since it is required that 
$n\Delta_n\to+\infty$ and $\Delta_n\to 0$, as $n$ tends to infinity, in order to obtain the break fractions consistency
when the number of breaks is known. 
%
\end{remark}
}

\textcolor{black}{
\begin{remark}
\corr{In practice, $c$ has to be chosen in order to use the top 
right part of the matrix of observations to estimate the parameter $\mu_0^\star$. 
This choice can either come from a prior biological knowledge or from a simple visualization of
the data. In the case of the analysis of HiC data, the size of the interaction
diagonal blocks are expected to be small compared with the size of the
chromosome \textit{i.e.} the size of the data matrix. In this context,
$c=3/4$ can be safely chosen, as suggested in \cite{levy2014two}.} 
If the value of $c$ is misspecified, the estimator of $\mu_0^\star$ is biased. The consistency result of Theorem 1 still holds
if (A3) is replaced by $\underset{1\leq k\leq K^\star}{\min}\left|\mu_k^\star-\mathbb{E}(\bar{Y}_{G_{01}})\right|>0$.
\end{remark}
}


\begin{proof}[Sketch of proof of Theorem \ref{theo:theo}]
In order to prove (\ref{eq:theo}), we shall prove that 
\textcolor{black}{
$\Prob\left(\Khat<\Ks\right)$ and $\Prob\left(\Khat>\Ks\right)$
tend to zero as $n$ tends to infinity. Note that
$$
\Prob\left(\Khat<\Ks\right)\leq\sum_{K=1}^{\Ks -1}\P\left(\Khat=\K\right)
\textrm{ and }\Prob\left(\Khat>\Ks\right)\leq\sum_{K=\Ks+1}^{\Kmax}\P\left(\Khat=\K\right).
$$
Hence, we shall prove that
}
for $\K<\Ks$ and $\K>\Ks$,
\begin{equation*}
\P\left(\Khat=\K\right)\longrightarrow 0,\textrm{ as } n\to+\infty.
\end{equation*}
\corr{Observe that by definition of $\Khat$ given in (\ref{eq:def:Khat}), 
\begin{align*}
\Prob\left(\Khat=\K\right)&\leq\Prob\left(\underset{\bt\in\AnKDelta}{\min}\Qn^{\K}(\bt)-\underset{\bt\in\AnKsDelta}{\min}\Qn^{\K^\star}(\bt)\leq
0\right)\\
&\leq
\Prob\left(\underset{\bt\in\AnKDelta}{\min}\Qn^{\K}(\bt)-\Qn^{\K^\star}(\bts)\leq
0\right),
\end{align*}
since, for large enough $n$, $\Deltan\leq \Deltataus$, and hence
  $\bts$ belongs to $\AnKsDelta$. Thus, we shall focus on}
\begin{equation*}
\Prob\left(\underset{\bt\in\AnKDelta}{\min}\Jn(\bt)\leq0\right),
\end{equation*}
where
\begin{equation}\label{eq:def_Jn}
\Jn(\bt)=\frac{2}{\n(\n+1)}\left(\Qn^{\K}(\bt)-\Qn^{\K^\star}(\bts)\right),
\end{equation}
We shall prove in \textcolor{black}{the supplementary material} that
\begin{equation}\label{eq:dec:Jn}
\Jn(\bt)=\Kn(\bt)+\Vn(\bt)+\Wn(\bt)+\Zn(\bt),
\end{equation}
where $\Kn$, $\Vn$, $\Wn$ and $\Zn$ are defined by (\ref{eq:Kn_Vn_Wn}), (\ref{eq:KD_Kz}), (\ref{eq:VD_Vz}), (\ref{eq:WD_Wz})
and (\ref{eq:Z}) \textcolor{black}{in Section \ref{Section: appendix}}.
In (\ref{eq:dec:Jn}), $\Kn$ corresponds to the deterministic part and the other terms correspond to the random part of $\Jn$.


The remainder of the proof is based on Lemma~\ref{Lemme:MinorationKn}, which is proved in Section \ref{Annexe:MinorationKn} and
which provides a lower bound for the deterministic part of $\Jn$,
and on Lemmas \ref{lem:maj:Vn}, \ref{lem:maj:Wn} and \ref{Lemme:InegaliteConcentrationZn}, 
\textcolor{black}{given in Section \ref{Section: appendix}},
which provide deviation inequalities for the random terms of $\Jn$.

\begin{lemma}\label{Lemme:MinorationKn}
Let $\Kn(\bt)$ be defined by (\ref{eq:Kn_Vn_Wn}) and (\ref{eq:KD_Kz}),
then
\begin{enumerate}[(i)]
\item if $K<K^{\vrai}$,
\begin{equation*} 
\min_{\bt \in \AnK^{1/\n}} \Kn(\bt)\geq \frac{\lambdainfz^2}{64}\left(\Deltataus\right)^4,
\end{equation*}
\item if $K>K^{\vrai}$,
\begin{equation*}
\min_{\bt \in \AnKDelta} \Kn(\bt)\geq  \frac{\lambdainfz^2}{4}\Deltan^2,
\end{equation*}
\item if $K=K^{\vrai}$, for all positive $\delta$,
\begin{equation} \label{Eq:MinorationKn-vrai}
\min_{\ensemblelemme} \Kn(\bt)\geq \frac{\lambdainfz^2}{32}
 \min\left( \Deltataus/2 ,\delta \right)\left(\Deltataus\right)^3,
\end{equation}
\end{enumerate}
where $\Deltataus$ is defined in (A\ref{Hypothese:TailleMaxBlocks}), $\lambdainfz$ is defined in (A\ref{Hypothese:mon_rupt}) and 
$\AnKDelta$ is defined in \eqref{eq:AnK}. $\AnK^{1/\n}$ is a particular case with $\Delta_n=1/n$ and
\begin{equation}\label{eq:norm_inf}
\left\|\bt-\bts\right\|_{\infty}=\max_{0\leq\kk\leq\Ks}\left|\tk-\tsk\right|.
\end{equation}
\end{lemma}

\noindent Thus,
\begin{eqnarray*}
\Prob\left(\underset{\bt\in\AnKDelta}{\min}\Jn(\bt)\leq0\right)
\leq\Prob\left(\underset{\bt\in\AnKDelta}{\min}\left[\Kn(\bt)+\Vn(\bt)+\Wn(\bt)+\Zn(\bt)\right]\leq0\right).
\end{eqnarray*}
The right hand side (rhs) of the previous inequality is bounded by
$$
\Prob\left(-\underset{\bt\in\AnKDelta}{\min}\Vn(\bt)-\underset{\bt\in\AnKDelta}{\min}\Wn(\bt)- \underset{\bt\in\AnKDelta}{\min}\Zn(\bt)
\geq\underset{\bt\in\AnKDelta}{\min}\Kn(\bt)\right).
$$
For bounding this term we shall use Lemma \ref{Lemme:MinorationKn} $(ii)$.
For $\K>\Ks$, we obtain
\begin{eqnarray}\label{eq:equation:fin:theo}
&&\Prob\left(\underset{\bt\in\AnKDelta}{\min}\Jn(\bt)\leq0\right)
    \leq\Prob\left(-\underset{\bt\in\AnKDelta}{\min}\Vn(\bt)\geq
    \frac{\lambdainfz^2}{12}\Deltan^2\right)\nonumber\\
    &+&\Prob\left(-\underset{\bt\in\AnKDelta}{\min}\Wn(\bt)\geq
     \frac{\lambdainfz^2}{12}\Deltan^2\right)
    +\Prob\left(- \underset{\bt\in\AnKDelta}{\min}\Zn(\bt)\geq
     \frac{\lambdainfz^2}{12}\Deltan^2\right).
\end{eqnarray}
By Lemmas \ref{lem:maj:Vn}, \ref{lem:maj:Wn} and \ref{Lemme:InegaliteConcentrationZn}, we conclude that
\begin{equation*}
\P\left(\Khat=\K\right)\underset{\n\rightarrow+\infty}{\longrightarrow} 0,
\end{equation*}
for $\K>\Ks$. The case $\K<\Ks$ can be proved by following the same lines.

\end{proof}

\corr{
\begin{remark}
We can observe from Theorem \ref{theo:theo} that adding a penalty term is not necessary for obtaining a consistent estimator of the number of diagonal blocks. This may be surprising since, in the one-dimensional case, it is proved in Theorem 9 of \cite{lavielle2000least} that
a penalty term is required. More precisely, the main difference between our two-dimensional framework and the one-dimensional case is the 
behavior of the deterministic part of our criterion $\Kn$:  it is lower bounded whatever the value
of $\K$ ($\K\geq\Ks$ or $\K<\Ks$), as proved in Lemma \ref{Lemme:MinorationKn}. On the contrary, in the one-dimensional case, a penalty term
of the type $\beta_n K$ is necessary to obtain such a lower bound when $\K\geq\Ks$. In the case where $\K<\Ks$, a lower bound for $\Kn$ is obtained
without penalization. For further details, see the proof of Theorem 9
in \cite{lavielle2000least}.
\end{remark}
}

\begin{theo}\label{theo: consistenceInstantRupt}
Assume that the assumptions of Theorem~\ref{theo:theo} hold then, for all $\deltaa>0$,
\begin{equation}\label{eq:consistancerupture}
\P\left(\left\|\bts-\bthatKhat\right\|_{\mathcal{H}}>\n\deltaa\right)\underset{\n\rightarrow+\infty}{\longrightarrow} 0,
\end{equation}
where $\bthatKhat$ is defined in~\eqref{def:Qn} and~(\ref{eq:def:Khat}) and $\|\cdot\|_{\mathcal{H}}$ denotes the Hausdorff distance defined by
\[\left\|\bts-\bthatK\right\|_{\mathcal{H}}=\max\left[\max_{0\leq\kk\leq\Ks}\min_{0\leq\el\leq\K}\left|\tsk-\thatl\right|, \max_{0\leq\el\leq\K}\min_{0\leq\kk\leq\Ks}\left|\tsk-\thatl\right|\right].\]
\end{theo}

\textcolor{black}{Observe that \eqref{eq:consistancerupture} can be rewritten as
$\P\left(\left\|\btaus-\widehat{\boldsymbol{\tau}}_{\Khat}\right\|_{\mathcal{H}}>\deltaa\right)
\underset{\n\rightarrow+\infty}{\longrightarrow} 0$, where $\widehat{\boldsymbol{\tau}}_{\Khat}=\bthatKhat/n$.}


\begin{proof}[Sketch of proof of Theorem \ref{theo: consistenceInstantRupt}]
Observe that
\begin{multline*}
\P\left(\left\|\bts-\bthatKhat\right\|_{\mathcal{H}}>\n\deltaa\right)
=\P\left(\left\{\left\|\bts-\bthatKhat\right\|_{\mathcal{H}}>\n\deltaa\right\}\cap
\left\{\Khat\neq\Ks\right\}\right)\\
+\P\left(\left\{\left\|\bts-\bthatKhat\right\|_{\mathcal{H}}>\n\deltaa\right\}\cap
\left\{\Khat=\Ks\right\}\right)\leq\P\left(\Khat\neq\Ks\right) +\P\left(\|\bthat_{\Ks}-\bts\|_\infty >\n\deltaa\right)
\end{multline*}
where $\|\bthat_{\Ks}-\bts\|_\infty$ is defined in (\ref{eq:norm_inf}) since $\|\bthatKhat-\bts\|_\infty=\|\bthatKhat-\bts\|_{\mathcal{H}}$
when $\Khat=\Ks$.
By Theorem~\ref{theo:theo}, proving (\ref{eq:consistancerupture}) amounts to proving that
$$
\P\left(\max_{0\leq\kk\leq\Ks}\left|\tsk-\thatk\right|>\n\deltaa\right)\to 0,\textrm{ as }
n\to+\infty .
$$
Observe that
\begin{equation*}
\P\left(\max_{1\leq\kk\leq\Ks}\left|\tsk-\thatk\right|>\n\deltaa\right)\leq\Prob\left(\underset{\ensemblelemmestar}{\min}\Jn(\bt)\leq0\right).
\end{equation*}
Using the same arguments as those used in the proof of Theorem \ref{theo:theo}, the proof follows from
the decomposition of $\Jn$ given by (\ref{eq:dec:Jn}), the lower bound (\ref{Eq:MinorationKn-vrai}) of
Lemma \ref{Lemme:MinorationKn} and the deviation inequalities for the random terms given by Lemmas \ref{lem:maj:Vn}, \ref{lem:maj:Wn} and
\ref{Lemme:InegaliteConcentrationZn}.
\end{proof}



\section{Numerical experiments}\label{Section: NumExp}


\corr{The goal of this section is to illustrate the theoretical results obtained in Section \ref{Section: TheoResults}.
For an application of our method to real data, we refer the reader to \cite{levy2014two}.}

\subsection{\corr{Simulation framework}}

We generated Gaussian diagonal block matrices according to Model~(\ref{eq:model}) with $\muks=1$ for the $\Ks=5$ diagonal blocks and $\muzs=0$ for different values of $\n$ $\left(\n\in\{500,1500\}\right)$. The change-point locations are $\left(\taus_0,\ldots,\taus_5\right)=(0,0.07,0.2,0.4,0.67,1)$ hence $\Deltataus=0.07$. We shall use different values for the standard deviation $\sigma$ of the $\epsij$: $\sigma\in\{1,...,10\}$. For each case, $500$ matrices were simulated and the procedure was tested. Examples of such matrices are displayed in Figure~\ref{fig:Matrices} for different values of $\sigma$.

\corr{The results that are presented below have been obtained by using the R package \textsf{HiCseg} which is available on the CRAN.
In this package, the values of $\Delta_n$ and $c$ are fixed and equal to $2/n$ and $3/4$, respectively. }


\begin{figure}[!h]
\begin{tabular}{cc}
$\sigma=1$&$\sigma=4$\\
\begin{minipage}[c]{0.45\textwidth}
\includegraphics[width=\linewidth]{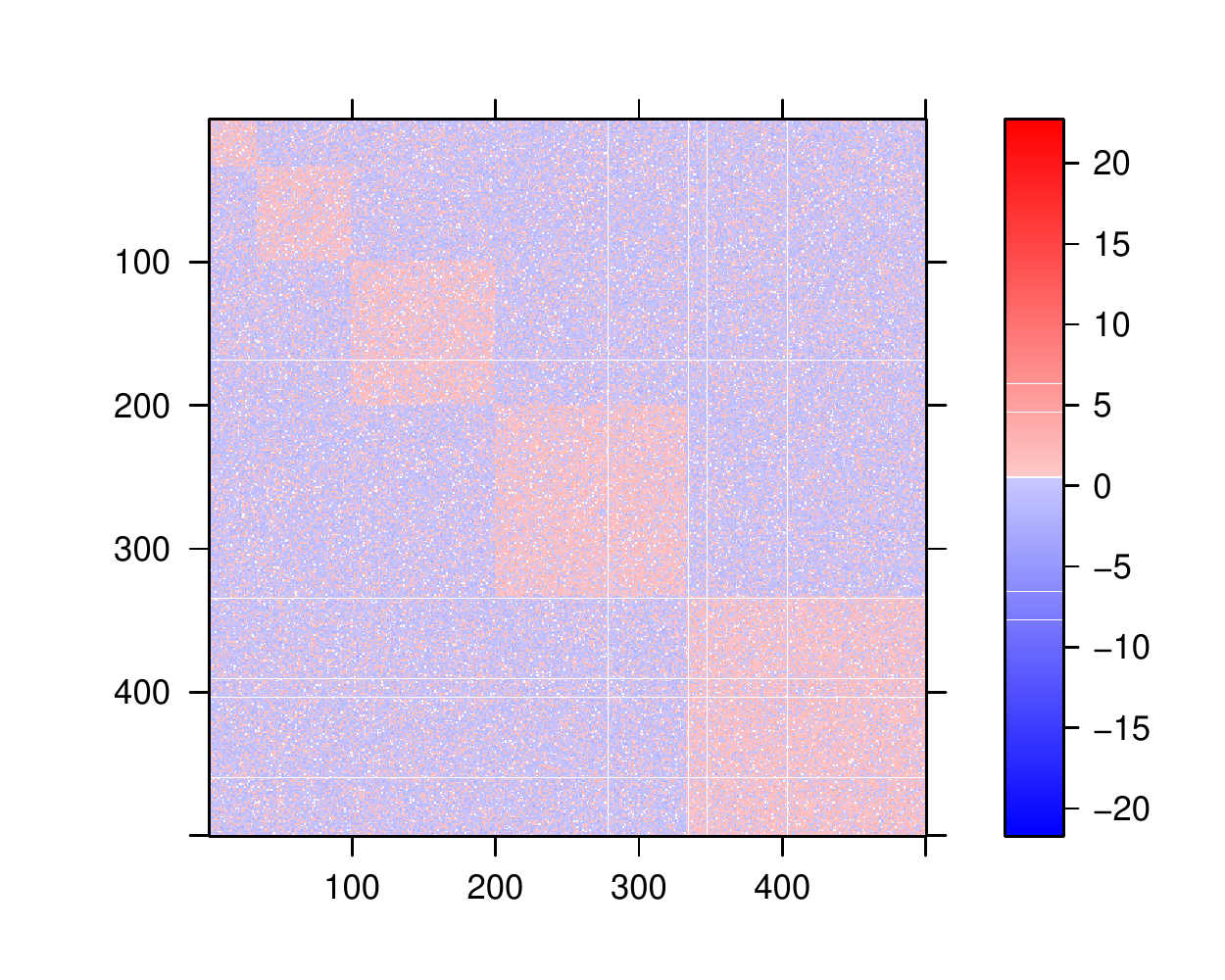}
\end{minipage}&
\begin{minipage}[c]{0.45\textwidth}
\includegraphics[width=\linewidth]{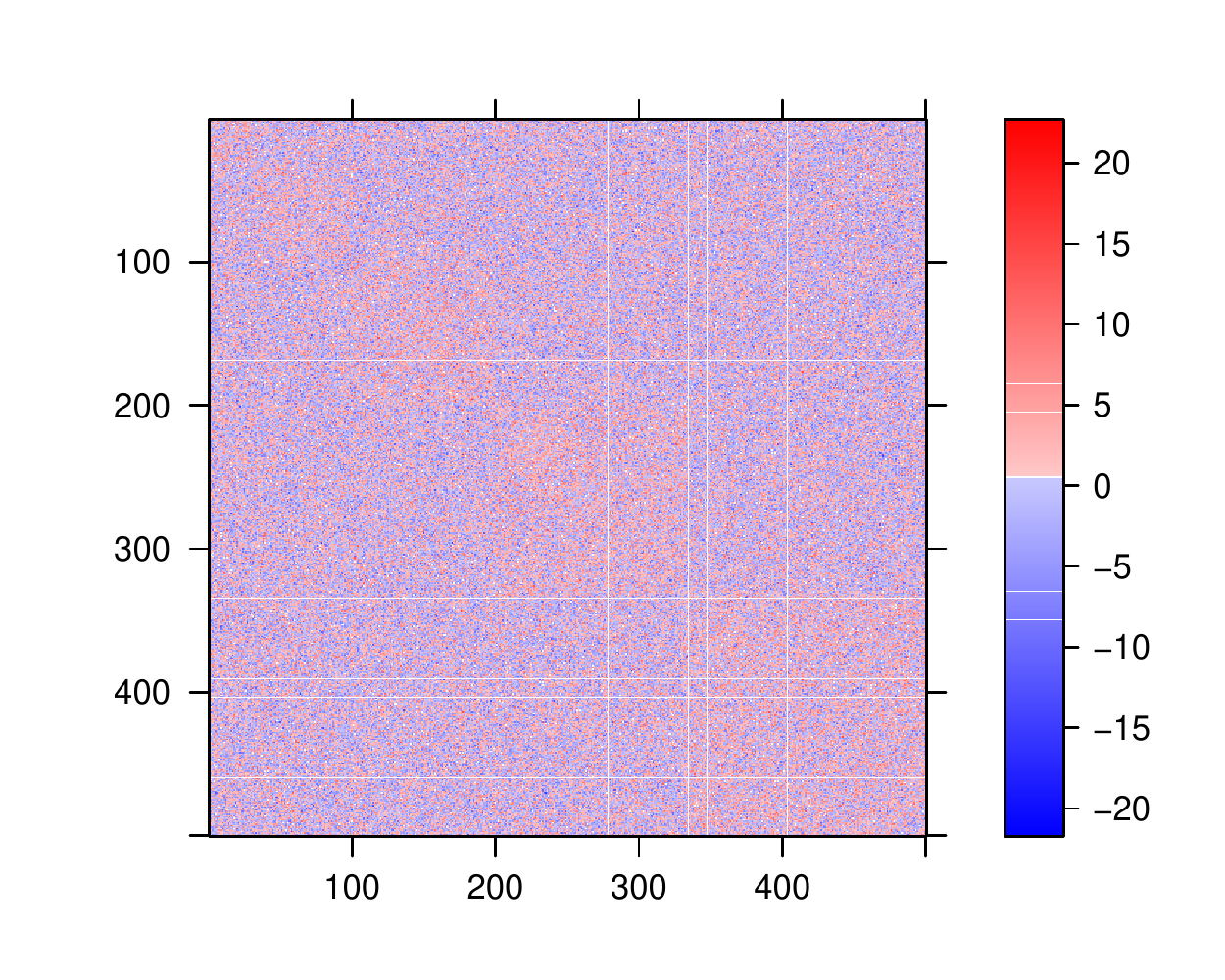}
\end{minipage}\\
\end{tabular}
   \caption{\label{fig:Matrices}Examples of simulated matrices following Model~(\ref{eq:model}) with $\left(\taus_0,\ldots,\taus_5\right)=(0,0.07,0.2,0.4,0.67,1)$ and $\n=500$ for two values of $\sigma$: $\sigma=1$ (left) and $\sigma=4$ (right).}
\end{figure}

\subsection{\corr{Statistical performance}}

\subsubsection{\corr{Performance of the statistical procedure}}

 \corr{We first consider the problem of estimating the true number of blocks
 $\Ks$, and provide some insight about the consistency of our
 procedure without penalty, 
outlined in Remark \ref{rem1}. The median, 1st and 3rd quantiles of
the estimated number of change points are displayed in Figure
\ref{fig:ResExperience} for $n$ in $\{500, 1500\}$ and for different
values of $\sigma$.}


\corr{On the one hand, we observe that for high signal to noise ratios, the true value of $\Ks$ is retrieved by our procedure.
On the other hand, when the signal to noise ratio becomes very low, $\Ks$ is not properly estimated.
In this situation, $\Ks$ is overestimated, which is in accordance with what occurs in the one-dimensional case where a non-penalized
procedure would result in a systematic overestimation of $\Ks$. However, when $n$ increases, the value of $\sigma$ from which this overestimation 
occurs is unsurprisingly larger.}\ \\

\begin{figure}[!h]
\begin{tabular}{ccc}
&$\n=500$&$\n=1500$\\
$\Khat$&\begin{minipage}[c]{0.45\textwidth}
\includegraphics[width=\linewidth]{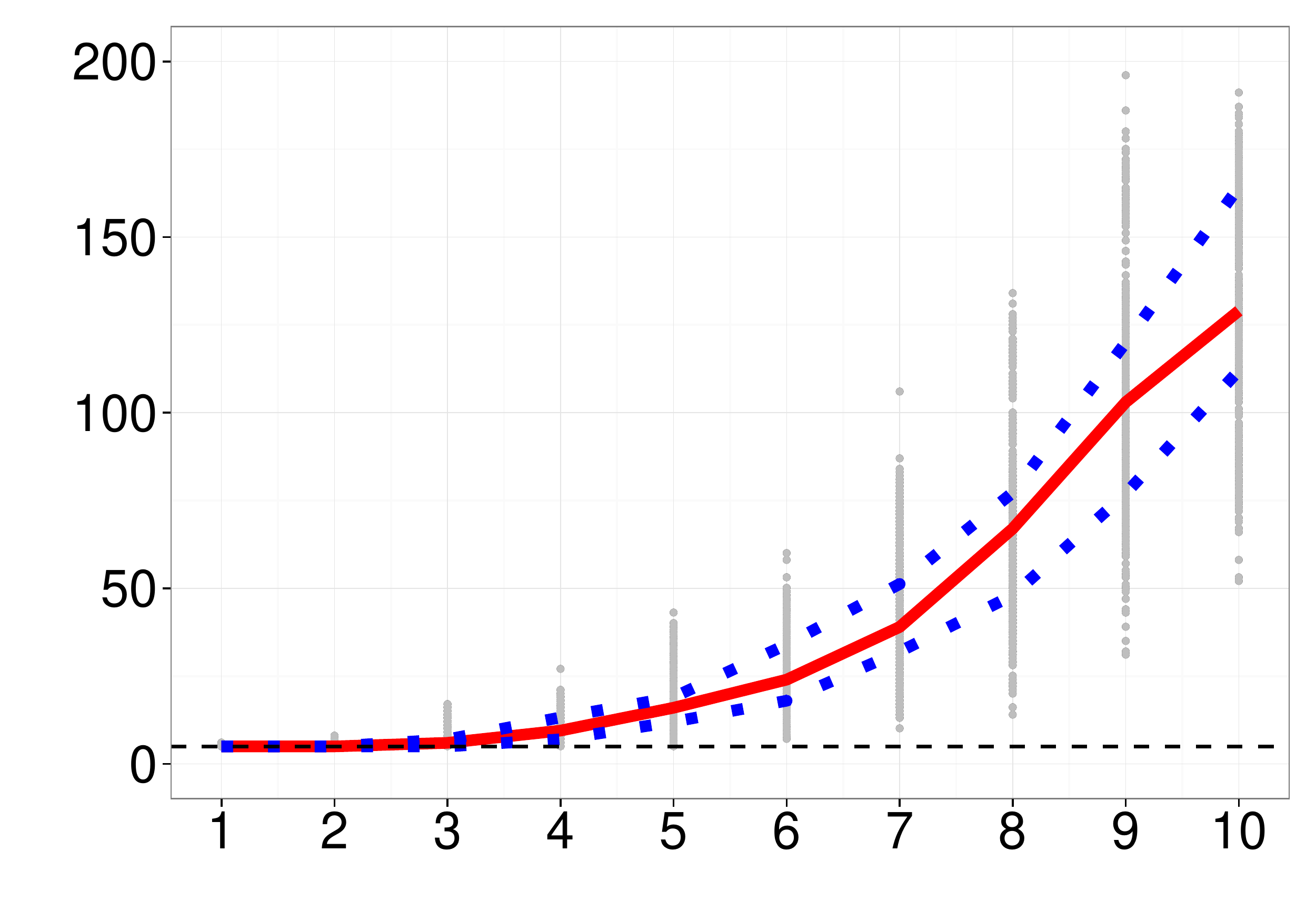}
\end{minipage}&
\begin{minipage}[c]{0.45\textwidth}
\includegraphics[width=\linewidth]{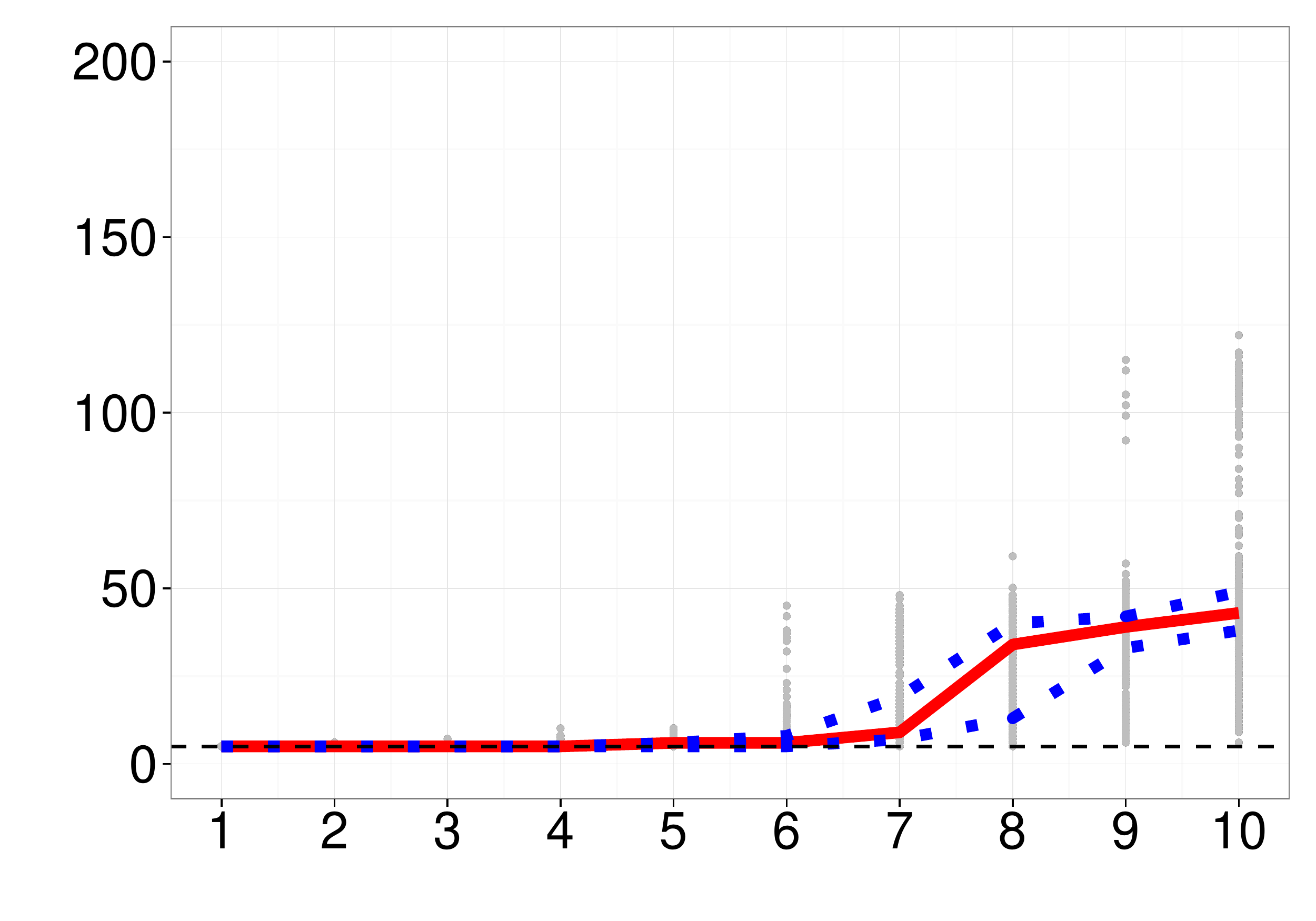}
\end{minipage}\\
&$\sigma$&$\sigma$\\
$\Khat$&\begin{minipage}[c]{0.45\textwidth}
\includegraphics[width=\linewidth]{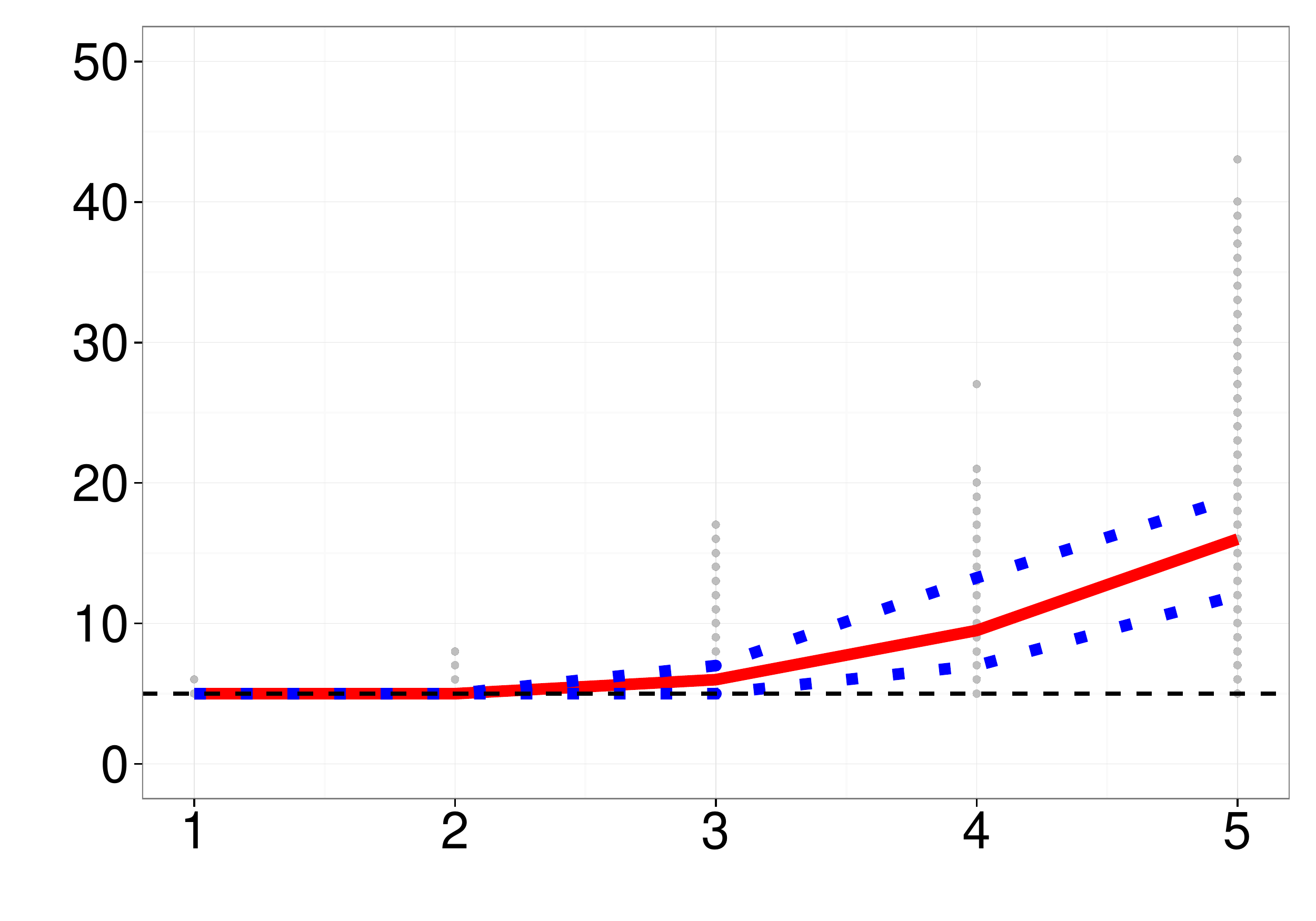}
\end{minipage}&
\begin{minipage}[c]{0.45\textwidth}
\includegraphics[width=\linewidth]{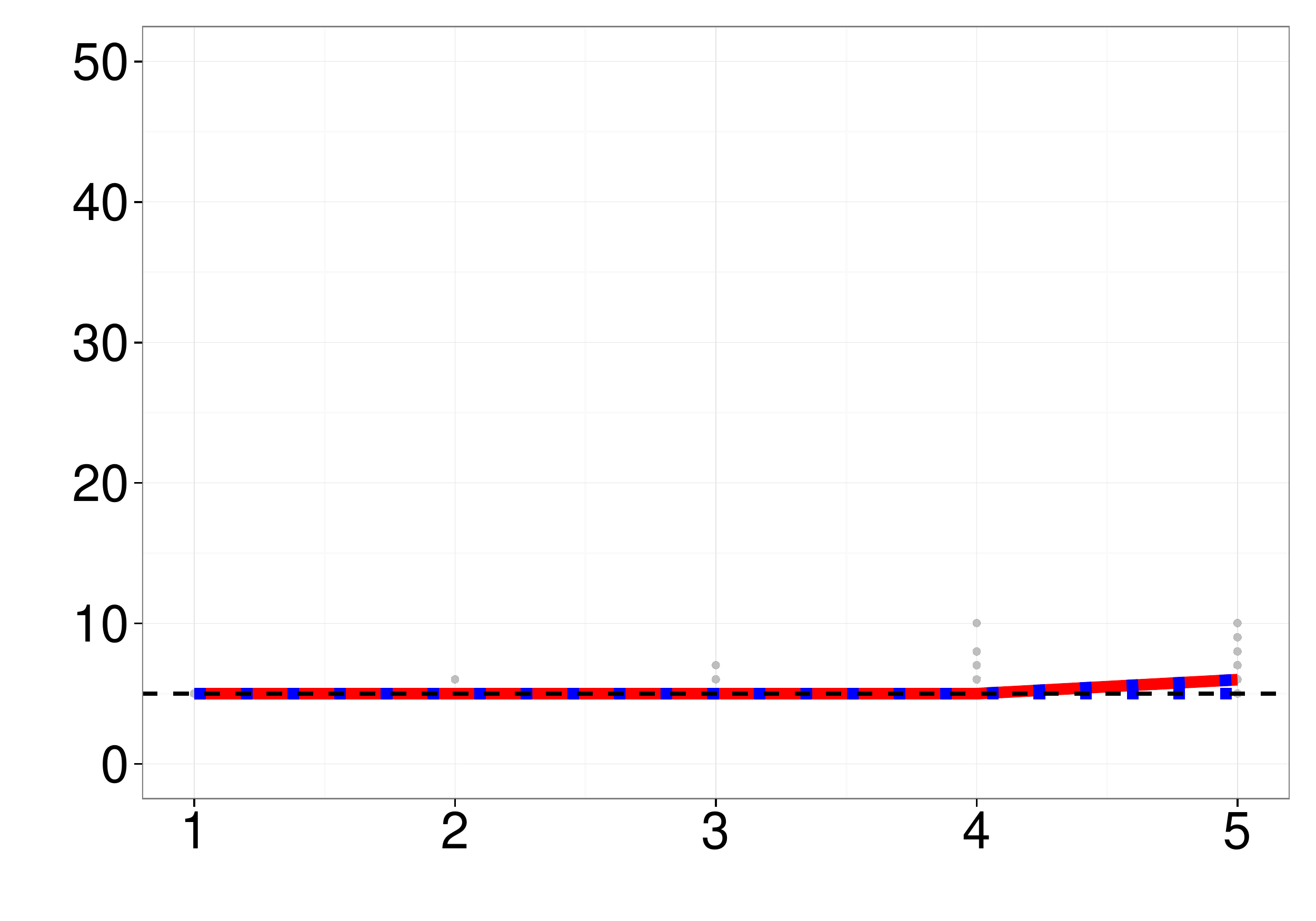}
\end{minipage}\\
&$\sigma$&$\sigma$\\
\end{tabular}
    \caption{\label{fig:ResExperience} \corr{Top: Median (plain), 1st and
      3rd quartiles (dotted line) of the estimations of $\Ks=5$ as a
      function of the standard deviation $\sigma$ for $\n=500$ (left)
      and $\n=1500$ (right). The values of $\Khat$ at each simulation
      are displayed with light grey dots. The dashed
        line corresponds to the true value of $\Ks$. Bottom: Same plots with the $x$-axis values restricted to $\{1,\dots,5\}$.}}
\end{figure}


\corr{To illustrate the performance of our procedure in terms of the
  estimation of change-point location, Figure~\ref{fig:Haus} displays the boxplots of the two parts
of the Hausdorff distance defined by:
\begin{eqnarray}
\label{eq:hausdorff}\left\|\bts-\bthatKhat\right\|_{\mathcal{H}^1}&=&\max_{0\leq\kk\leq\Ks}\min_{0\leq\el\leq\Khat}\left|\tsk-\thatl\right|,\\
\nonumber\left\|\bts-\bthatKhat\right\|_{\mathcal{H}^2}&=&\max_{0\leq\el\leq\Khat}\min_{0\leq\kk\leq\Ks}\left|\tsk-\thatl\right|.
\end{eqnarray}
We observe from this figure that when $\Ks$ is overestimated, the true change-points are recovered ($\|\cdot\|_{\mathcal{H}^1}$ is close to $0$), 
the other estimated change-points being spurious ones ($\|\cdot\|_{\mathcal{H}^2}$ is large). As proved in Theorem~\ref{theo: consistenceInstantRupt},
this phenomenon is less \corr{visible} when $n$ becomes large.}


\begin{figure}[!h]
\begin{tabular}{ccc}
&$\n=500$&$\n=1500$\\
\rotatebox{90}{\hspace{1.5em}\small $\mathcal{H}^1$}&
\begin{minipage}[c]{0.45\textwidth}
\includegraphics[width=\linewidth]{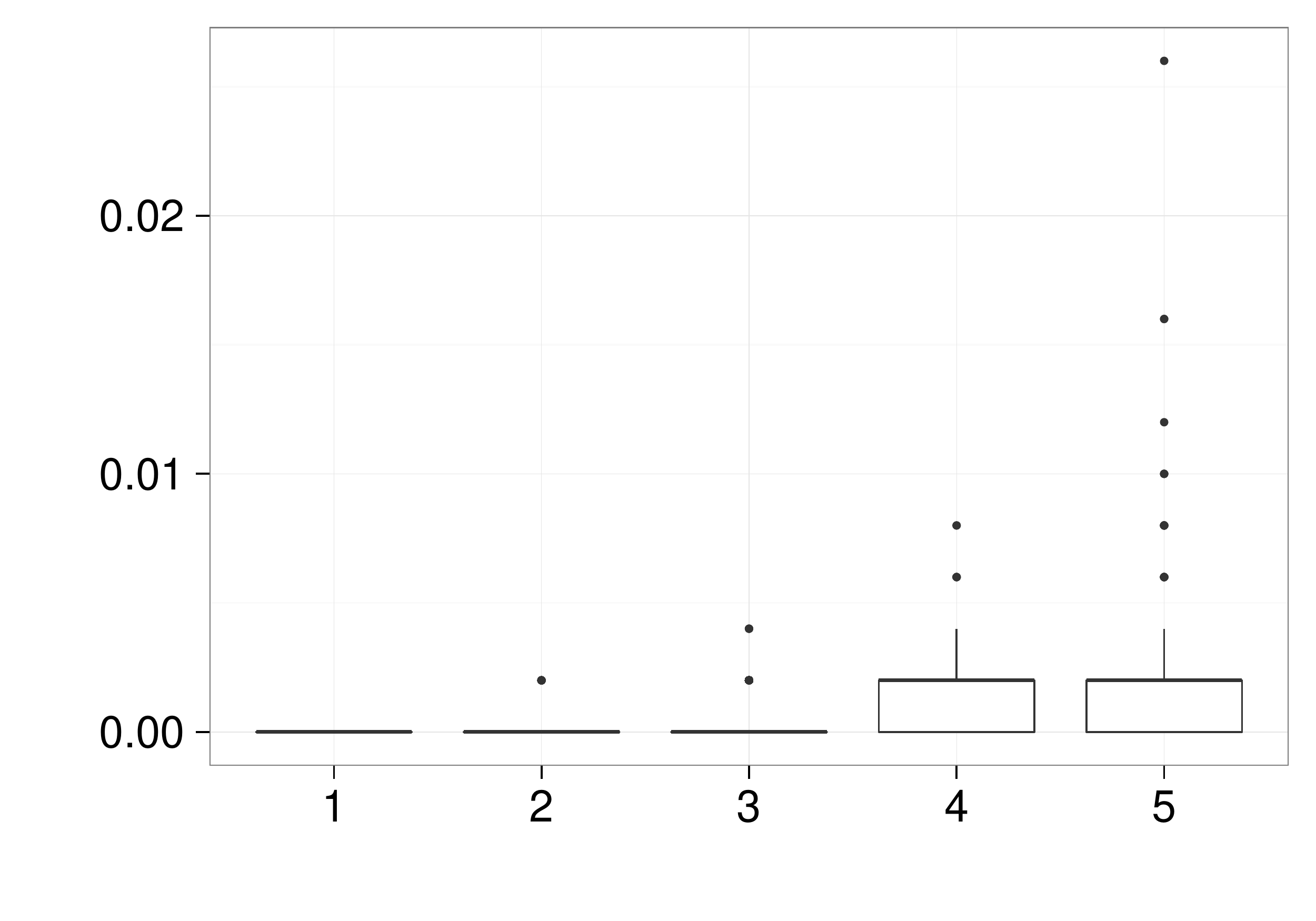}
\end{minipage}&
\begin{minipage}[c]{0.45\textwidth}
\includegraphics[width=\linewidth]{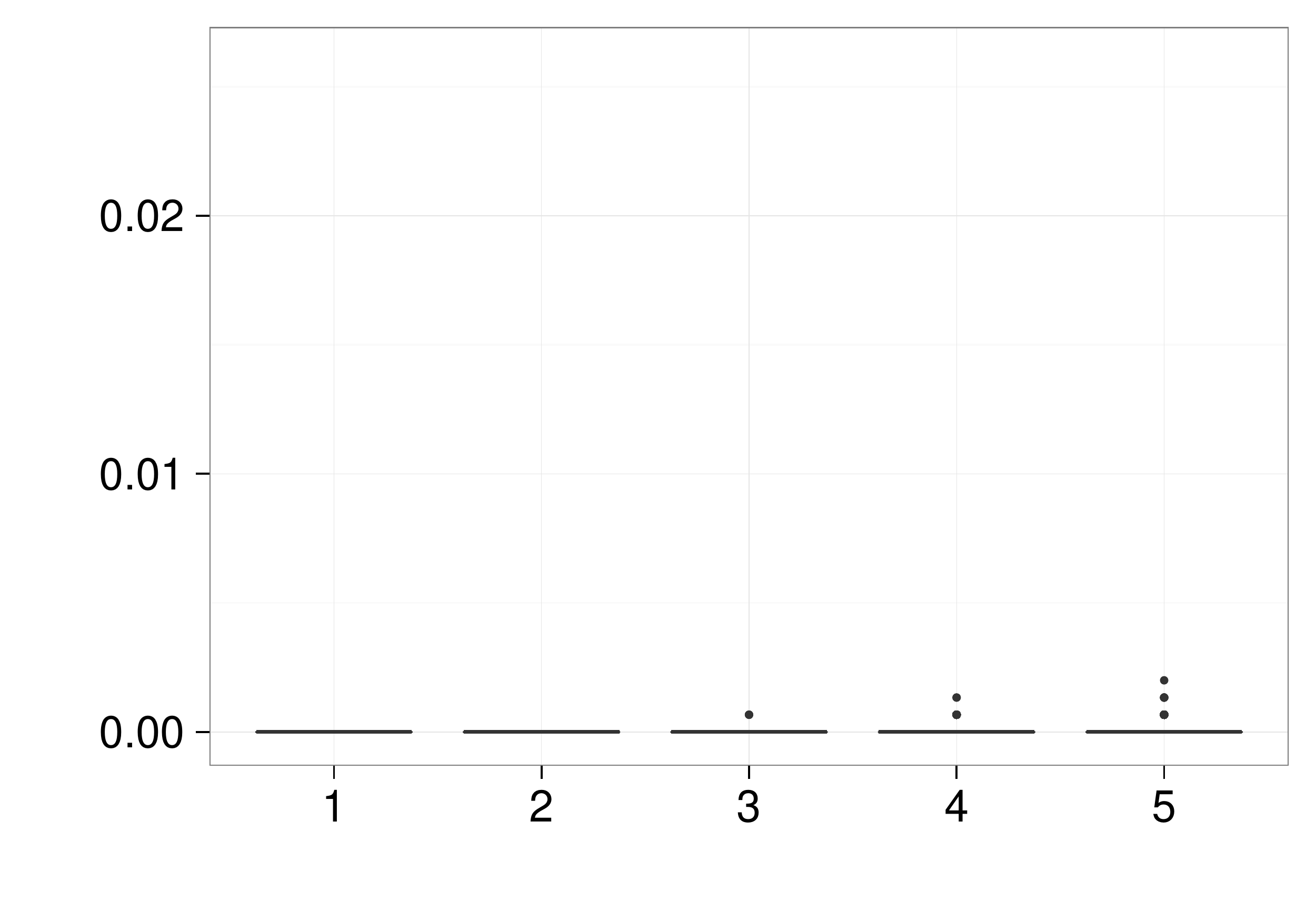}
\end{minipage}\\
\rotatebox{90}{\hspace{1.5em}\small $\mathcal{H}^2$}&
\begin{minipage}[c]{0.45\textwidth}
\includegraphics[width=\linewidth]{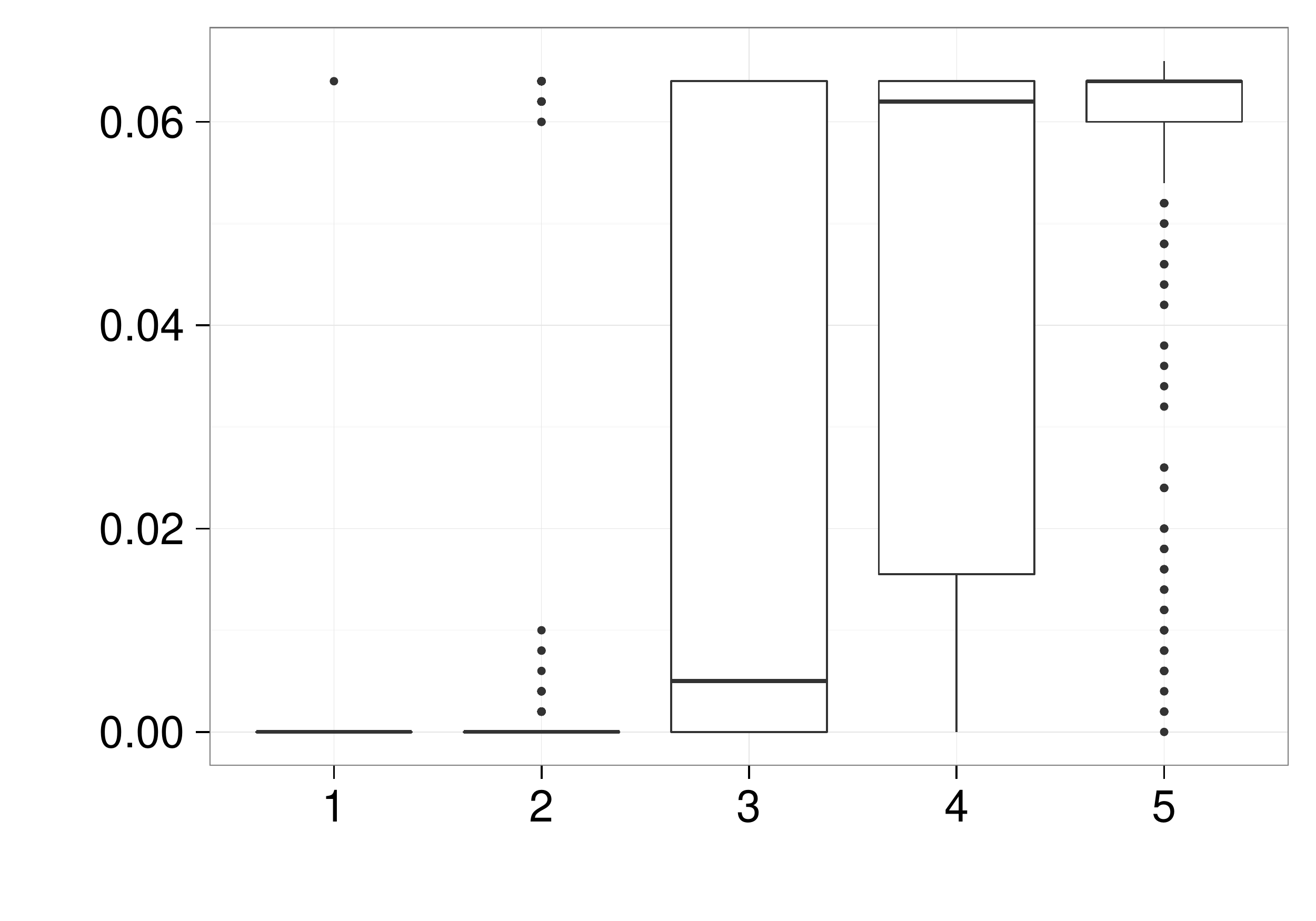}
\end{minipage}&
\begin{minipage}[c]{0.45\textwidth}
\includegraphics[width=\linewidth]{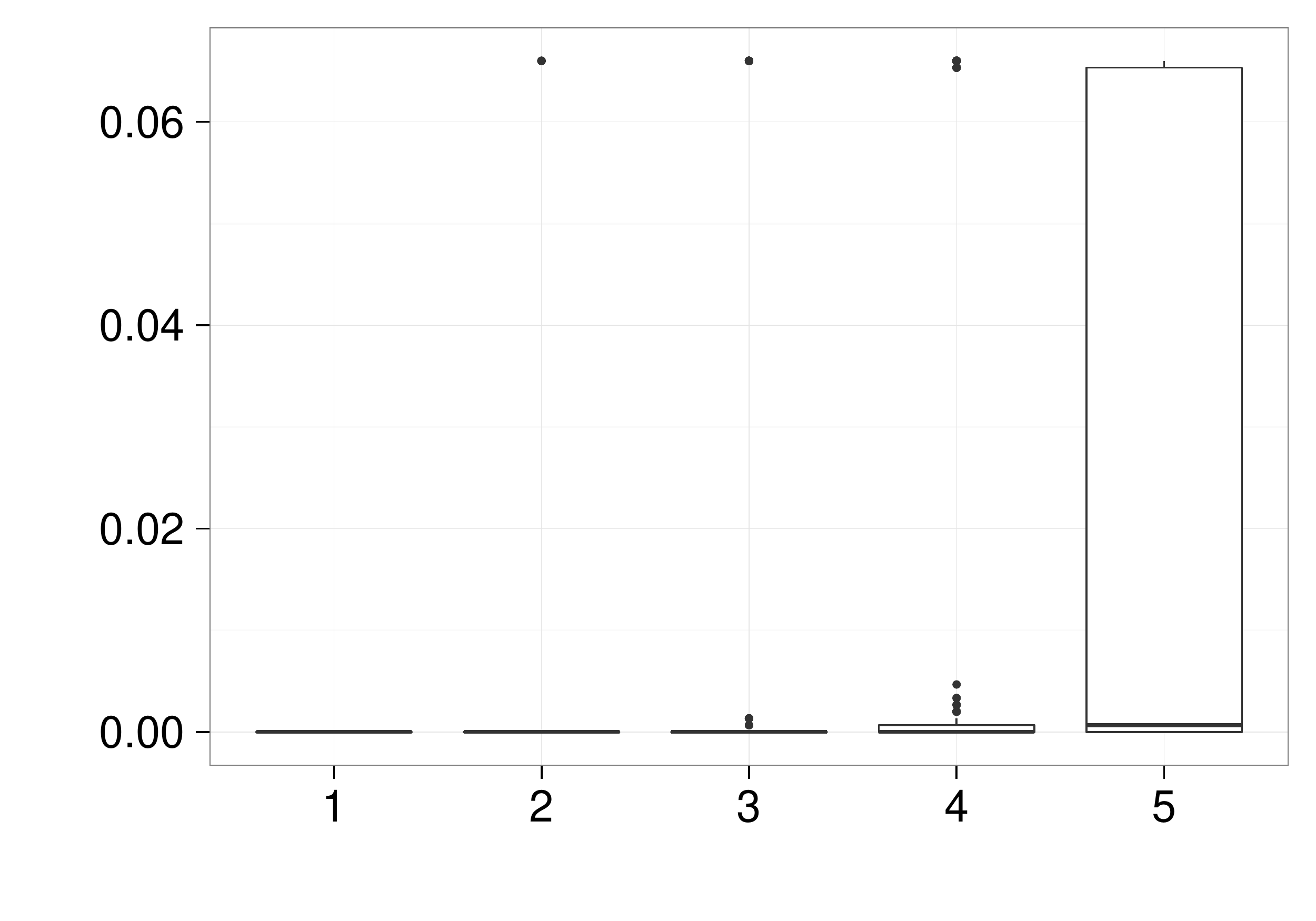}
\end{minipage}\\
\end{tabular}
    \caption{\label{fig:Haus} \corr{Boxplots of the two parts of the
        Hausdorff distance: $\mathcal{H}^1$ (top) and $\mathcal{H}^2$
        (bottom) for $\n=500$ (left) and $\n=1500$ (right). For each
        case, the boxplots are displayed as a function of $\sigma$.}}
\end{figure}

\subsubsection{\corr{Effect of a poor estimation of $\mu_0^\star$}}

\corr{We study the behavior of our segmentation procedure when $\mu_0^\star$ is poorly estimated which may occur, for instance, when 
the constant $c$ appearing in (\ref{eq:G01}) is too small. 
To this end, we generated data in which 
the mean of the $n_0\times n_0$ top right part of the observation matrix is modified, where $n_0$ is defined in (\ref{Hypothese:TailleMaxBlocks}).
More precisely, the mean of this part is equal to $\mu_0^\star+\omega$, where $\omega\in\{0.2,0.4,0.6,0.8\}$. The results are displayed in
Figure \ref{fig:mu0}. We can see from this figure that when the value of $\mu_0^\star+\omega$ is close to the values of
the means of the diagonal blocks our procedure tends to overestimate $\Ks$. This phenomenon is less visible when $n$ is large.}

\begin{figure}[!h]
\begin{tabular}{ccc}
&$\n=500$&$\n=1500$\\
\rotatebox{90}{\hspace{1.5em}\small $\sigma=1$}&
\begin{minipage}[c]{0.45\textwidth}
\includegraphics[width=\linewidth]{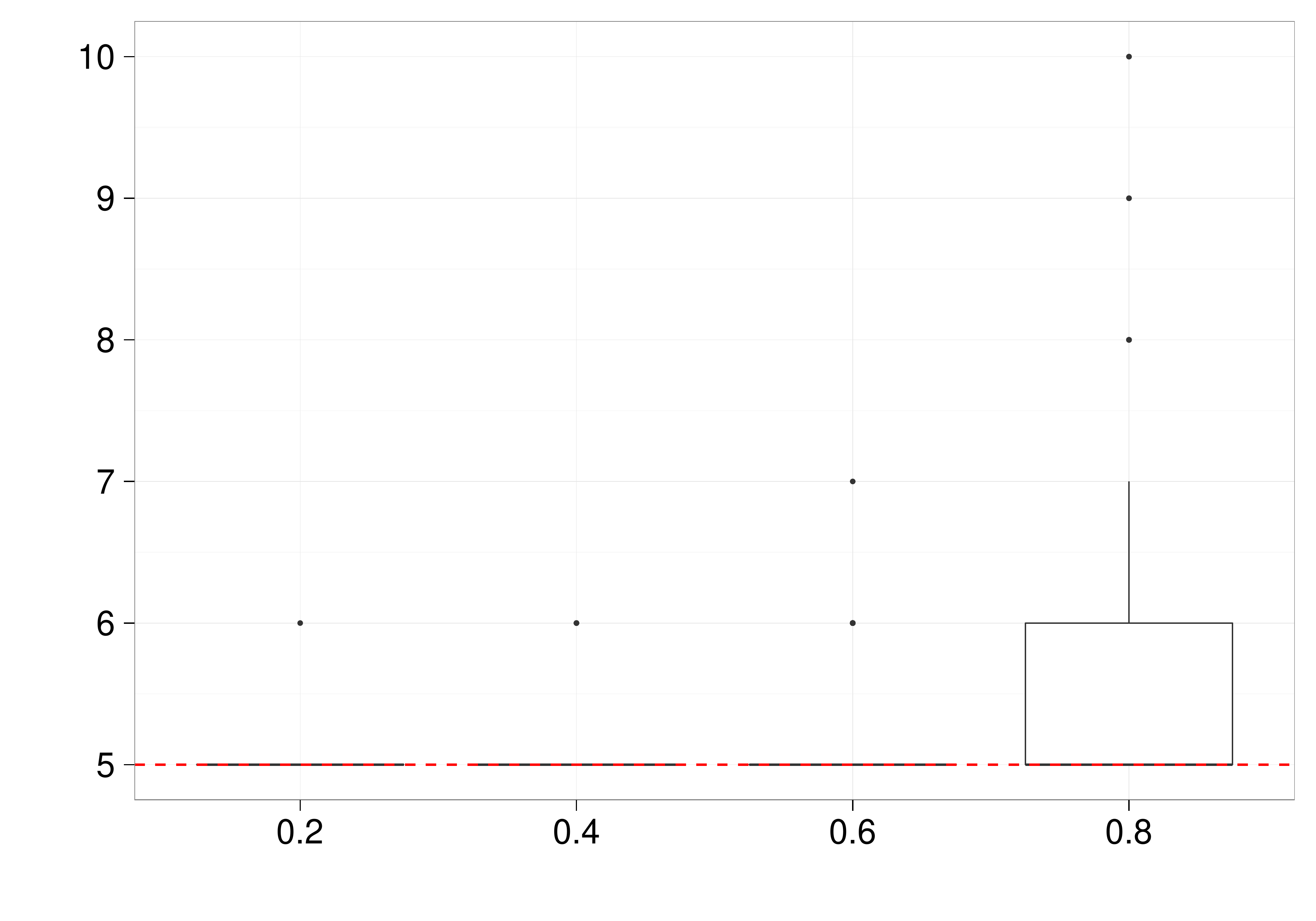}
\end{minipage}&
\begin{minipage}[c]{0.45\textwidth}
\includegraphics[width=\linewidth]{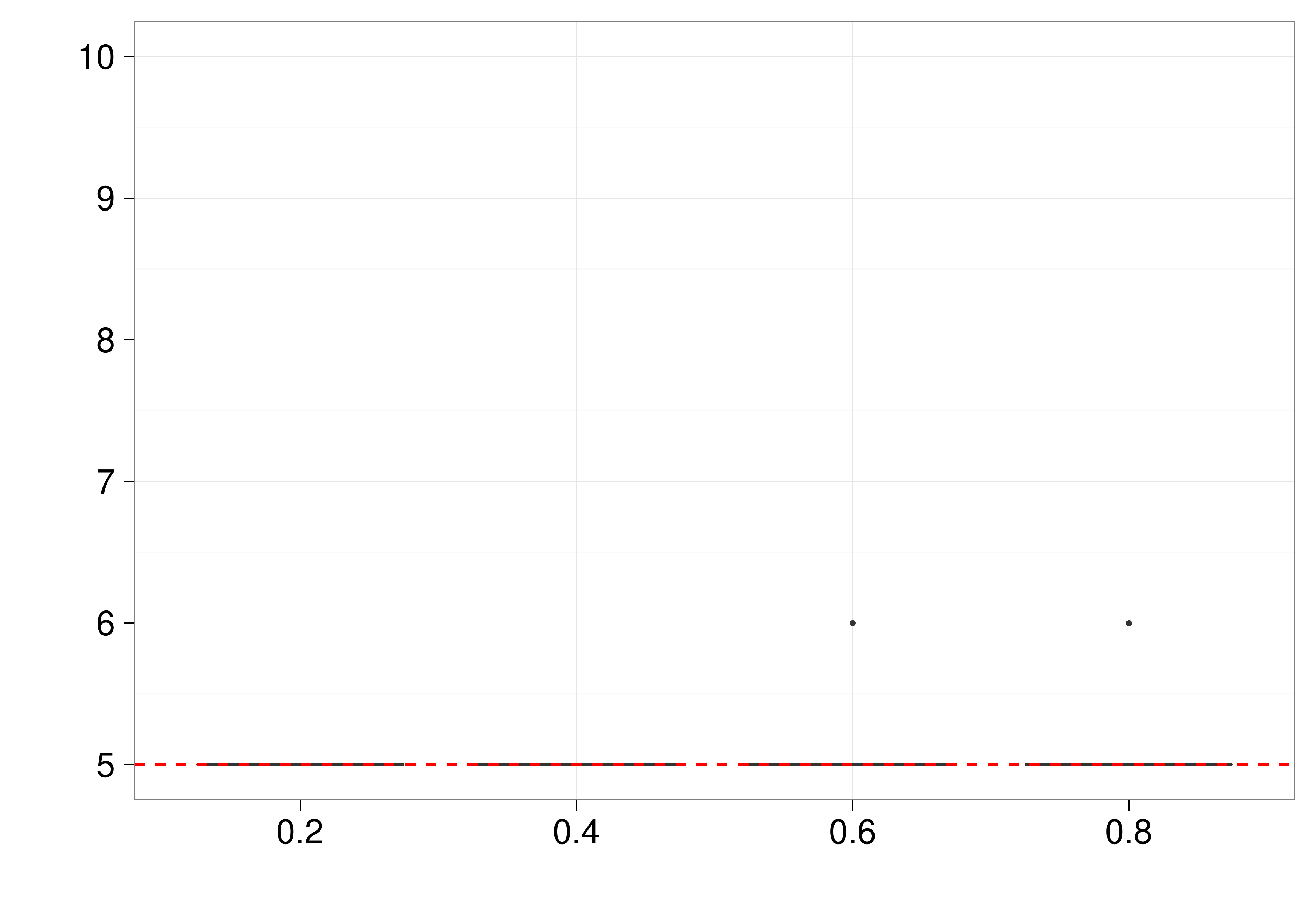}
\end{minipage}\\
& $\omega$ &  $\omega$\\
\rotatebox{90}{\hspace{1.5em}\small $\sigma=4$}&
\begin{minipage}[c]{0.45\textwidth}
\includegraphics[width=\linewidth]{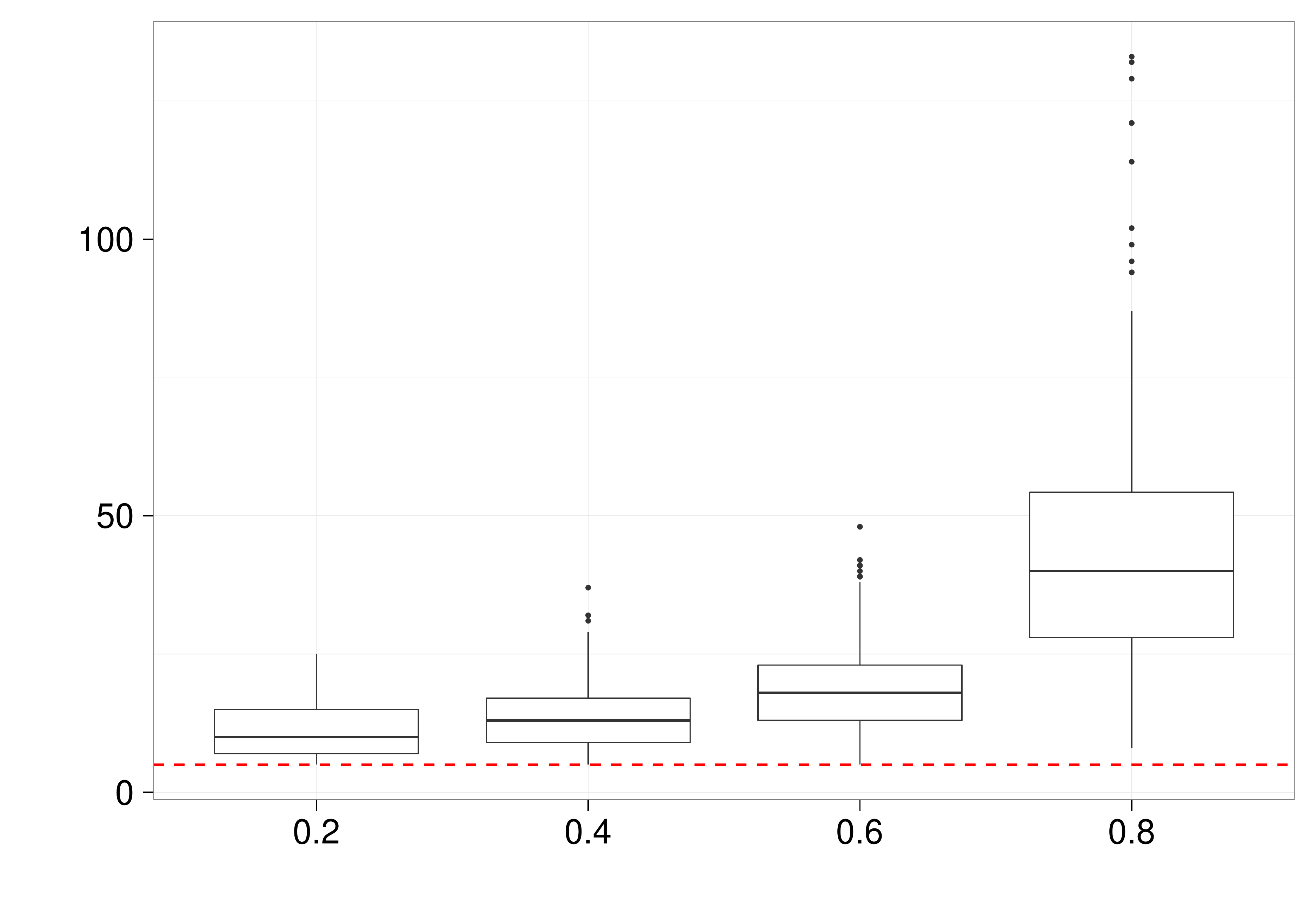}
\end{minipage}&
\begin{minipage}[c]{0.45\textwidth}
\includegraphics[width=\linewidth]{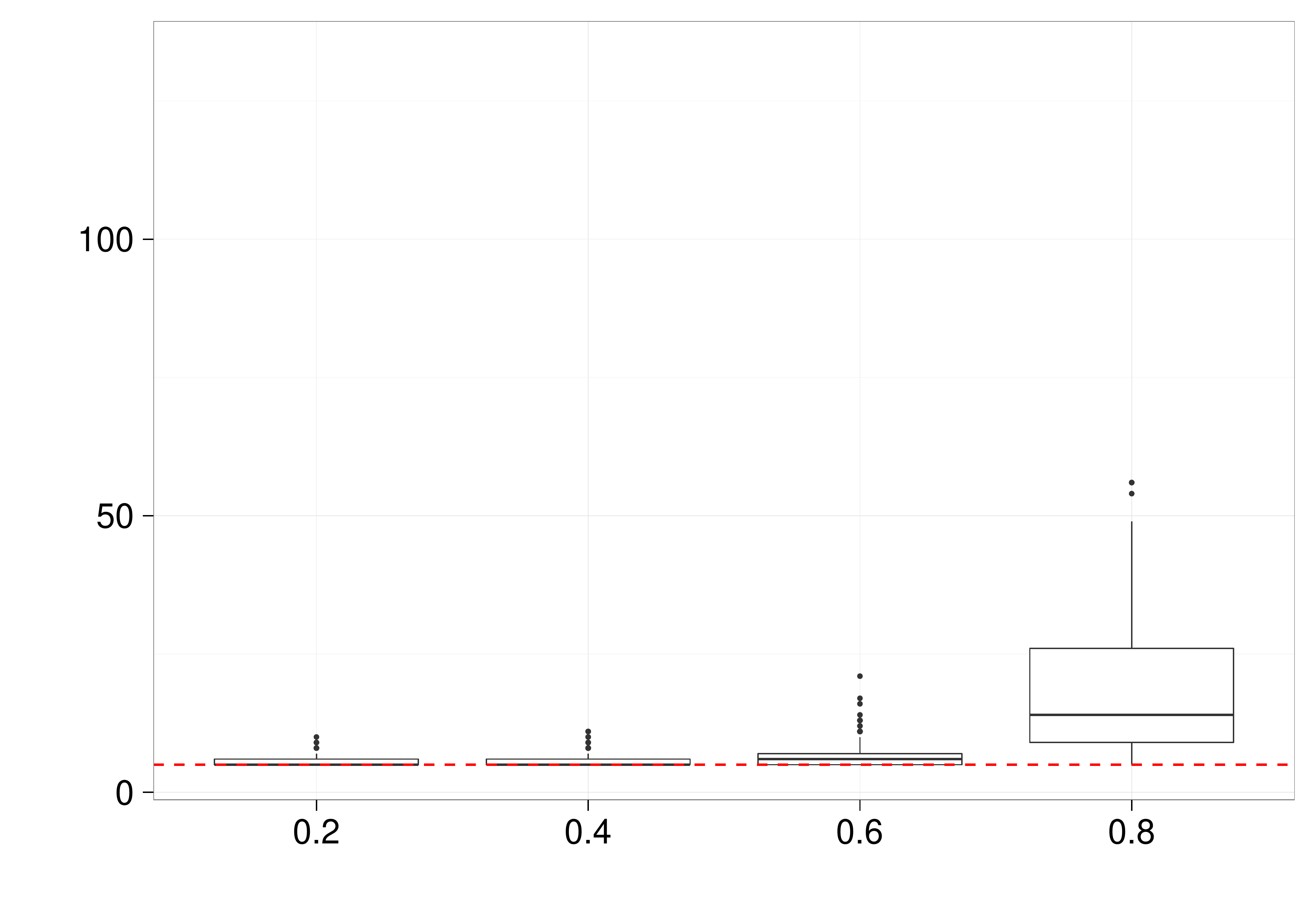}
\end{minipage}\\
& $\omega$ &  $\omega$\\
\end{tabular}
    \caption{\label{fig:mu0} \corr{Boxplots of $\Khat$ for $\sigma=1$ (top) and $\sigma=4$ (bottom) for $\n=500$ (left) and $\n=1500$ (right). 
For each case, the boxplots are displayed as a function of $\omega$.}}
\end{figure}


\section{Discussion}\label{Section:Discussion}

\textcolor{black}{
In this paper, we established that the (slightly modified) least-squares estimators for the number of blocks and their boundaries
in a block diagonal matrix are consistent. 
\corr{Note that the obtained results are non standard in the sense that we
proved that penalizing the least-squares criterion is not required to
obtain a consistent estimator of the number of diagonal blocks. This has
to be contrasted with the one-dimensional case, where it is well-known
that a penalization is required to ensure consistency, see for instance \cite{lavielle2000least}. More
precisely, a close look at the proof of Theorem 9 in \cite{lavielle2000least} shows that a
penalty is required to discard models such that $K>\Ks$. This comes
from the fact that in the one-dimensional setting when $K>\Ks$ the
deterministic part $\Kn$ of $\Jn$ vanishes for all segmentations $\bt$
satisfying $\|\bts-\bt\|_{\mathcal{H}^1}=0$ (\textit{i.e.} for all segmentations $\bt$
nested in the true segmentation $\bts$). This bias term being null, a
penalty term has to be added to the criterion to compensate
the stochastic deviations of the random terms in $\Jn$. In the
two-dimensional setting, the deterministic part $\Kn$ does not vanish
when $K>\Ks$ --as proved in Lemma \ref{Lemme:MinorationKn}-- ensuring consistency.}}


\textcolor{black}{
The framework that we have chosen for proving our results consists in assuming that the observations are independent
and that the size of the observation matrix is large (asymptotic framework),
which is adapted to the analysis of HiC experiments.
From a practical point of view, the independence assumption is not always satisfied, for instance when the observation matrix
is a correlation or a similarity matrix, see for example \cite{dehman:2015,delatola:2015}. Hence, relaxing the independence assumption
to retrieve diagonal block boundaries in such cases would be a natural extension of this paper.
%
\corr{Moreover,  it could be interesting to see if a penalty term needs to be added to our criterion in order to retrieve properly
the break fractions in a non asymptotic setting. This  will be the subject of a future work.}
}





\section{Proofs}\label{Section: appendix}
\subsection{Definition of $\Kn$, $\Vn$, $\Wn$ and $\Zn$}
We define hereafter $\Kn$, $\Vn$, $\Wn$ and $\Zn$ which appear in (\ref{eq:dec:Jn}) by:
\begin{equation}\label{eq:Kn_Vn_Wn}
\Kn(\bt)=\KDn(\bt)+\Kzn(\bt),\; \Vn(\bt)=\VDn(\bt)+\Vzn(\bt),\; \Wn=\WDn(\bt)+\Wzn(\bt),
\end{equation}
and
\begin{multline}\label{eq:KD_Kz}
\KDn(\bt)=\frac{2}{\n(\n+1)}\left(\sum_{\kk=1}^{\K}{\sum_{(\ii,\jj)\in\Dk}{\left(\mbE\left[\Yij\right]-\mbE\left[\YbarDk\right]\right)^2}}\right),\\
\Kzn(\bt)=\frac{2}{\n(\n+1)}\sum_{(\ii,\jj)\in\Gzero}{\left(\mbE\left[\Yij\right]-\mbE\left[\YbarGun\right]\right)^2},
\end{multline}
\begin{multline}\label{eq:VD_Vz}
\VDn(\bt)=\frac{2}{\n(\n+1)}\left[\sum_{\kk=1}^{\Ks}{\frac{\left(\sum_{(\ii,\jj)\in\Dsk}{\epsij}\right)^2}{\left|\Dsk\right|}} -\sum_{\kk=1}^{\K}{\frac{\left(\sum_{(\ip,\jp)\in\Dk}{\epsipjp}\right)^2}{\left|\Dk\right|}}\right],\\
\Vzn(\bt)=\frac{2}{\n(\n+1)}\frac{1}{|\Gun|^2}\left(\sum_{(\ii,\jj)\in\Gun}{\epsij}\right)^2\left(|\Gzero|-|\Gszero|\right),
\end{multline}
\begin{multline}\label{eq:WD_Wz}
\WDn(\bt)=\frac{4}{\n(\n+1)}\left[\sum_{\kk=1}^{\Ks}{\left(\sum_{(\ii,\jj)\in\Dsk}{\epsij}\right)\muks}-\sum_{\kk=1}^{\K}{\left[ \left(\sum_{(\ip,\jp)\in\Dk}{\epsipjp}\right)\mbE\left[\YbarDk\right]\right]}\right],\\
\Wzn(\bt)=\frac{4}{\n(\n+1)}\muzs\left(\sum_{(\ii,\jj)\in\Gszero}{\epsij}-\sum_{(\ii,\jj)\in\Gzero}{\epsij}\right),
\end{multline}
\begin{equation}\label{eq:Z}
\Zn(\bt)=\frac{4}{\n(\n+1)}\frac{1}{|\Gun|}\left(\sum_{(\ii,\jj)\in\Gun}{\epsij}\right) \left[\left(\sum_{(\ii,\jj)\in\Gszero}{\epsij}-\sum_{(\ii,\jj)\in\Gzero}{\epsij}\right)-\sum_{(\ii,\jj)\in\Gzero}(\mbE\left[\Yij\right]-\muzs)\right].
\end{equation}
In the equations, $\Gszero$ and $\Gun$ are defined in (\ref{eq:G00:G01}) and (\ref{eq:G01}) and $\Gzero$ has the same definition
as $\Gszero$ except that $\bts$ is replaced by $\bt$.
\subsection{Proof of Lemma \ref{Lemme:MinorationKn}}\label{Annexe:MinorationKn}

We shall first rewrite $\KDn$ and $\Kzn$ defined by (\ref{eq:KD_Kz}).
Let us first denote by
\begin{equation}\label{eq:comptage_nkl}
\nkl=\left|\Dk\cap\Dsl\right|,
\end{equation}
the number of observations that belong to the intersection of the
two blocks $\Dk$ and $\Dsl$ (with the convention that
$\D_{0}=\Gzero$ and $D_0^\star=\Gzero^\star$) and
$$\nk=\sum_{\el=0}^{\K^{\vrai}}{\nkl} \quad\textrm{and}\quad
\nsl=\sum_{\kk=0}^{\K}{\nkl}.$$
Since $\mbE\left[\YbarGun\right]=\muzs$, $\Gzero\subset\left(\bigcup_{\ell=0}^{\Ks} {D_\ell}^\star\right)$ and
$\mbE\left[\Yij\right]=\muks$, for all $(\ii,\jj)\in\Dsk$, we obtain
\begin{eqnarray}\label{eq:new_def_Kzn}
\Kzn(\bt)&=&\frac{2}{\n(\n+1)}\sum_{(\ii,\jj)\in\Gzero}{\left(\mbE\left[\Yij\right]-\muzs
\right)^2} = \frac{2}{\n(\n+1)}\sum_{\el=0}^{\K^{\vrai}}
\sum_{(\ii,\jj)\in
\Gzero\cap\Dsl} {\left(\mbE\left[\Yij\right]-\muzs \right)^2} \nonumber\\
&=& \frac{2}{\n(\n+1)}\sum_{\el=0}^{\K^{\vrai}} \nzl
{\left(\muls-\muzs \right)^2}.
\end{eqnarray}
Since $|D_k|=\sum_{\el=0}^{\K^{\vrai}} \left|\Dk\cap\Dsl\right|=\sum_{\el=0}^{\K^{\vrai}} {\nkl}=\nk$,
\begin{eqnarray}\label{eq:YbarDk}
\mbE\left[\YbarDk\right]&=&\frac{1}{\nk}\sum_{(\ii,\jj)\in\Dk}{\mbE\left[\Yij\right]}
=\frac{1}{\nk}\sum_{\el=0}^{\K^{\vrai}} \sum_{(\ii,\jj)\in\Dk \cap\Dsl} {{\mbE\left[\Yij\right]}}
=\frac{1}{\nk}\sum_{\el=0}^{\K^{\vrai}}{\muls\nkl},
\end{eqnarray}
where we use for all $\kk\in\{1,\ldots,\K\}$, $\Dk\subset\left(\bigcup_{\ell=0}^{\Ks} {D_\ell}^\star\right)$. Thus,
\begin{eqnarray*}
&&\sum_{(\ii,\jj)\in\Dk}{\left(\mbE\left[\Yij\right]-\mbE\left[\YbarDk\right]\right)^2}
= \frac{1}{\nk^2} \sum_{(\ii,\jj)\in\Dk}{\left(\nk
\mbE\left[\Yij\right] -\sum_{\lp=0}^{\K^{\vrai}} \mulps\nklp
\right)^2} \\
&=&\frac{1}{\nk^2}\sum_{\el=0}^{\K^{\vrai}}\sum_{(\ii,\jj)\in\Dk\cap\Dsl}{\left(\nk
\mbE\left[\Yij\right] -\sum_{\lp=0}^{\K^{\vrai}} \mulps\nklp
\right)^2}
=\frac{1}{\nk^2}\sum_{\el=0}^{\K^{\vrai}}{\nkl\left[\sum_{\lp=0}^{\K^{\vrai}}{\nklp\left(\muls-\mulps\right)}\right]^2}\\
&=&
\frac{1}{\nk^2}\sum_{\el=0}^{\K^{\vrai}}{\sum_{\lun=0}^{\K^{\vrai}}{\sum_{\ldeux=0}^{\K^{\vrai}}{\nkl\nklun\nkldeux\left(\muls-\muluns\right)
\left(\muls-\muldeuxs\right)}}}\\
&=& \frac{1}{\nk^2}\sum_{\el=0}^{\K^{\vrai}}\sum_{\lun=0}^{\K^{\vrai}}\nkl\nklun\left(\muls-\muluns\right)
\sum_{\ldeux=0}^{\K^{\vrai}}\nkldeux\left(\muls-\muldeuxs\right)\\
&=&\frac{1}{\nk}\sum_{\el=0}^{\K^{\vrai}}\sum_{\lun=0}^{\K^{\vrai}}\nkl\nklun\left({\muls}^2-\muluns\muls\right)
-\frac{1}{\nk^2}\sum_{\ldeux=0}^{\K^{\vrai}}\nkldeux\muldeuxs\underbrace{\sum_{\el=0}^{\K^{\vrai}}\sum_{\lun=0}^{\K^{\vrai}}\nkl\nklun\left(\muls-\muluns\right)}_{=0}\\
&=& \frac{1}{\nk}\sum_{\el=0}^{\K^{\vrai}}\sum_{\lun=0}^{\K^{\vrai}}\nkl\nklun\left({\muls}^2-\muluns\muls\right)
=\frac{1}{2\nk}\sum_{\el=0}^{\K^{\vrai}}{\sum_{\lp=0}^{\K^{\vrai}}{\nkl\nklp\left(\muls-\mulps\right)^2}}.
\end{eqnarray*}
Hence,
\begin{equation}\label{eq:new_def_KDn}
\KDn(\bt)=\frac{1}{\n(\n+1)}\sum_{\kk=1}^{\K}\frac{1}{\nk}\sum_{\el=0}^{\K^{\vrai}}{\sum_{\lp=0}^{\K^{\vrai}}{\nkl\nklp\left(\muls-\mulps\right)^2}}.
\end{equation}
\begin{figure}[!h]
\begin{center}
\begin{tabular}{cc}
\hspace{-15mm}\begin{tikzpicture}
\draw[opacity=0.4] (0,0) grid[step=0.5] (8,8);

\draw (1.25,8) node[above,color=blue]{$\ts_{\el-1}$};
\draw (0,6.75) node[left,color=blue]{$\ts_{\el-1}$};
\draw (4.25,8) node[above,color=blue]{$\tsl$};
\draw (0,3.75) node[left,color=blue]{$\tsl$};
\draw (7.25,8) node[above,color=blue]{$\ts_{\el+1}$};
\draw (0,0.75) node[left,color=blue]{$\ts_{\el+1}$};

\draw[ultra thick,color=blue,dotted] (0,8) rectangle (1,7);
\draw[fill=blue,opacity=0.2] (0,8) rectangle (1,7);

\draw[ultra thick,color=blue,dotted] (1,7) rectangle (4,4);
\draw[fill=blue,opacity=0.2] (1,7) rectangle (4,4);

\draw[ultra thick,color=blue,dotted] (4,4) rectangle (7,1);
\draw[fill=blue,opacity=0.2] (4,4) rectangle (7,1);

\draw[ultra thick,color=blue,dotted] (7,1) rectangle (8,0);
\draw[fill=blue,opacity=0.2] (7,1) rectangle (8,0);

\draw (2.25,8) node[above,color=red]{$\ttt_{\el-1}$};
\draw (0,5.75) node[left,color=red]{$\ttt_{\el-1}$};
\draw (6.25,8) node[above,color=red]{$\tl$};
\draw (0,1.75) node[left,color=red]{$\tl$};

\draw[ultra thick,color=red,dotted] (0,8) rectangle (2,5.5);
\draw[fill=red,opacity=0.2] (0,8) rectangle (2,5.5);

\draw[ultra thick,color=red,dotted] (2,5.5) rectangle (6,2);
\draw[fill=red,opacity=0.2] (2,5.5) rectangle (6,2);

\draw[ultra thick,color=red,dotted] (6,2) rectangle (8,0);
\draw[fill=red,opacity=0.2] (6,2) rectangle (8,0);

\draw (1.5,6.25) node{$\n_{\el-1,\el}$};
\draw (3,4.75) node{$\n_{\el,\el}$};
\draw (5,3) node{$\n_{\el,\el+1}$};
\draw (6.5,1.5) node{$\n_{\el+1,\el+1}$};

\draw (3,6.25) node{$\n_{0,\el}$};
\draw (5,4.75) node{$\n_{\el,0}$};
\draw (6.5,3) node{$\n_{0,\el+1}$};

\draw (6.5,6.5) node{$\n_{0,0}$};

\fill[color=gray,opacity=0.3] (0,0) -- (0,8) -- (8,0) -- cycle;
\draw[ultra thick] (0,0) rectangle (8,8);
\end{tikzpicture}&
\hspace{-5mm}\begin{tikzpicture}
\draw[opacity=0.4] (0,0) grid[step=0.5] (8,8);

\draw (1.75,8) node[above,color=blue]{$\ts_{\kk-1}=\ttt_{\el-1}$};
\draw (0,6.25) node[left,color=blue]{$\ts_{\kk-1}$};
\draw (5.25,8) node[above,color=blue]{$\tsk=\ttt_{\el+1}$};
\draw (0,2.75) node[left,color=blue]{$\tsk$};

\draw[ultra thick,color=blue,dotted] (0,8) rectangle (1.5,6.5);
\draw[fill=blue,opacity=0.2] (0,8) rectangle (1.5,6.5);

\draw[ultra thick,color=blue,dotted] (1.5,6.5) rectangle (5,3);
\draw[fill=blue,opacity=0.2] (1.5,6.5) rectangle (5,3);

\draw[ultra thick,color=blue,dotted] (5,3) rectangle (8,0);
\draw[fill=blue,opacity=0.2] (5,3) rectangle (8,0);

\draw (3.25,8) node[above,color=red]{$\tl$};
\draw (0,4.75) node[left,color=red]{$\tl$};

\draw[ultra thick,color=red,dotted] (0,8) rectangle (1.5,6.5);
\draw[fill=red,opacity=0.2] (0,8) rectangle (1.5,6.5);

\draw[ultra thick,color=red,dotted] (1.5,6.5) rectangle (3,5);
\draw[fill=red,opacity=0.2] (1.5,6.5) rectangle (3,5);

\draw[ultra thick,color=red,dotted] (3,5) rectangle (5,3);
\draw[fill=red,opacity=0.2] (3,5) rectangle (5,3);

\draw[ultra thick,color=red,dotted] (5,3) rectangle (8,0);
\draw[fill=red,opacity=0.2] (5,3) rectangle (8,0);

\draw (4,5.75) node{$\n_{0,\kk}$};
\draw (2.25,5.75) node{$\n_{\el,\kk}$};
\draw (0.75,7.25) node{$\n_{\el-1,\kk-1}$};
\draw (4,4) node{$\n_{\el+1,\kk}$};
\draw (6.5,1.5) node{$\n_{\el+2,\kk+1}$};
\draw (6.5,6) node{$\n_{0,0}$};

\fill[color=gray,opacity=0.3] (0,0) -- (0,8) -- (8,0) -- cycle;
\draw[ultra thick] (0,0) rectangle (8,8);
\end{tikzpicture}\\
  \end{tabular}
\end{center}
\caption{
Left: $K<\Ks$. Right: $K>\Ks$.
}
\label{fig:minorationKn}
\end{figure}
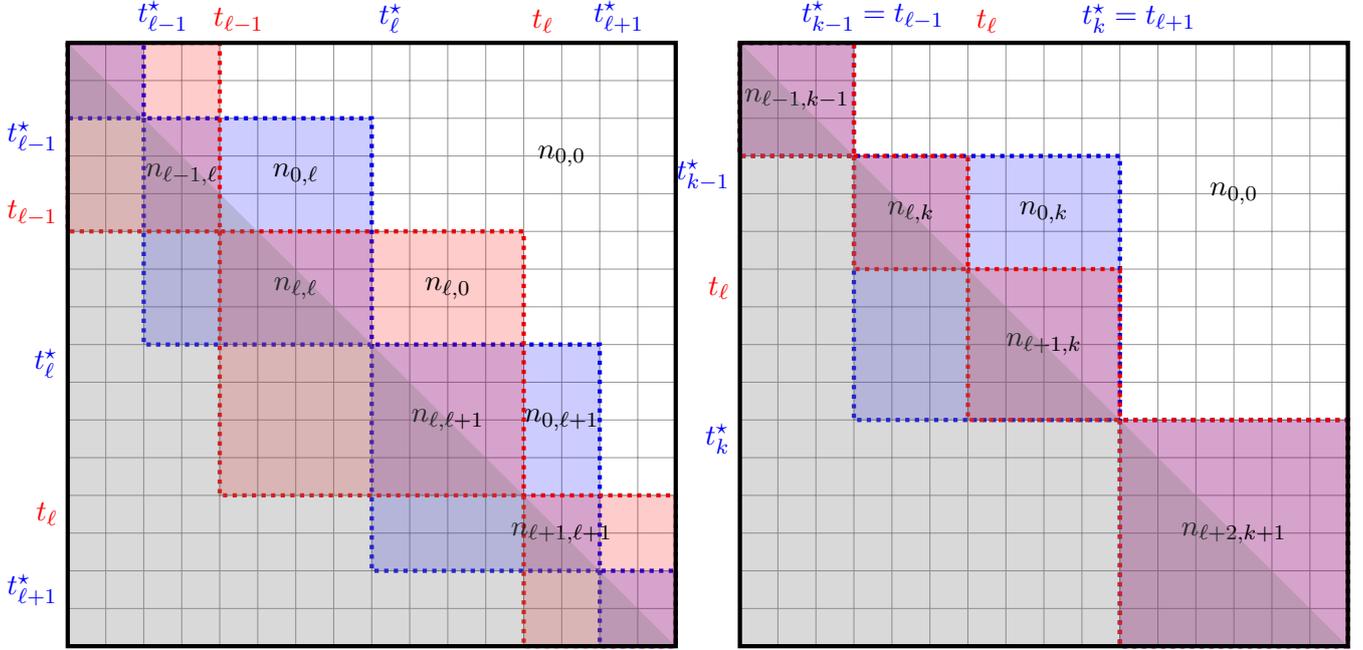


\subsubsection{Case $K<K^{\vrai}$ and $t \in \AnK^{1/\n}$.} Observe that $\Kn(\bt)\geq\KDn(\bt)$.
Since $K<\Ks$, $t_K-\ts_K=\ts_{\Ks}-\ts_K\geq n\Deltataus$. Hence, $\{k, t_k-\ts_k\geq n\Deltataus/2\}\neq\emptyset$.
Let $\ell=\min\{k, t_k-\ts_k\geq \Deltataus/2\}$, then $\el\geq1$ and
$$
t_{\ell-1}\leq\ts_\ell- n\Deltataus/2\leq \ts_\ell+ n\Deltataus/2\leq t_\ell .
$$
By definition of $\Deltataus$,
\begin{eqnarray}\label{eq:nlk}
\n_{\ell,\ell}&=&|\Dl\cap\Dsl|\geq\min\{(\ts_\ell-t_{\ell-1})(\ts_\ell-t_{\ell-1}+1)/2,(\tsl-\ts_{\ell-1})(\tsl-\ts_{\ell-1}+1)/2\}\nonumber\\
&\geq&\left(\n\Deltataus\right)^2/8,
\end{eqnarray}
and
\begin{eqnarray}\label{eq:nl0}
\n_{\el,0}\geq\min\{(\tl-\tsl)(\tsl-t_{\ell-1}),(\ts_{\ell+1}-\tsl)(\tsl-\ts_{\ell-1})\}
\geq\left(\n\Deltataus\right)^2/4.
\end{eqnarray}
Thus, 
using (\ref{eq:nlk}) and (\ref{eq:nl0}), we obtain
\begin{eqnarray*}
\Kn(\bt)&\geq&\frac{1}{\n(\n+1)\nl}\left[\n_{\el,\el}\n_{\el,0}\left(\mus_{0}-\mus_{\ell}\right)^2\right]
        \geq\frac{\lambdainfz^2}{\n(\n+1)\nl}\frac{\left(\n\Deltataus\right)^4}{32}
        \geq\frac{\left(\Deltataus\right)^4\lambdainfz^2}{64},
\end{eqnarray*}
since $\nl\leq n(n+1)/2$.\ \\

\subsubsection{Case $K>K^{\vrai}$ and $\bt \in \AnKDelta$.} 
We have
\begin{eqnarray*}
\Kn(\bt) \geq \Kzn(\bt) \geq \frac{2}{\n(\n+1)}  \n_{0,\kk}
\left(\mu_{\kk}^\star-\muzs \right)^2 \geq \frac{2}{\n(\n+1)}  \ \n_{0,\kk} \
\lambdainfz^2
\end{eqnarray*}
for any $k \in \left\{0,\ldots,\Ks\right\}$. Since $\bt\in\AnKDelta$, there exists $\el\in\{1,\ldots,\K-1\}$ such that for all $\kk\in\{0,\ldots,\Ks\}$
\[\left|\tsk-\tl\right|>\frac{\n\Deltan}{2},\]
(otherwise, it will imply that $\K\leq\Ks$). Moreover, let us choose $\kk$ such that $\ts_{\kk-1}+\n\Deltan/2<\tl<\tsk-\n\Deltan/2$ then
\[\n_{0,\kk}\geq\left(\tl-\ts_{\kk-1}\right)\left(\tsk-\tl\right)\geq\left(\frac{\n\Deltan}{2}\right)^2.\]
This leads to
\[\Kn(\bt)\geq\frac{1}{4}\lambdainfz^2\Deltan^2.\]


\subsubsection{Case $K = K^{\vrai}$ and $t \in \AnK^{1/\n}, \left\|\bt-\bts\right\|_{\infty} > \n\delta $.}

We have
\begin{equation}\label{eq:Bborneinf}
B_n(t)\geq\frac{1}{n(n+1)}\frac{1}{n_{\el}}n_{\el,\el^\prime}n_{\el,0} (\mu_0^{\star}-\mu_{\el^\prime}^{\star})^2
\end{equation}
for every $\el \in \left\{1,\ldots,K\right\}$ and every $\el^\prime \in \left\{1,\ldots,K^{\star}\right\}$. Then, we shall consider two cases:
{\it i)} $\left\|\bt-\bts\right\|_{\infty} < \frac{\n\Deltataus}{2}$ and {\it ii)} $\left\|\bt-\bts\right\|_{\infty} \geq \frac{\n\Deltataus}{2}$.\\

{\it i)} $\left\|\bt-\bts\right\|_{\infty} < \frac{\n\Deltataus}{2}$.\\

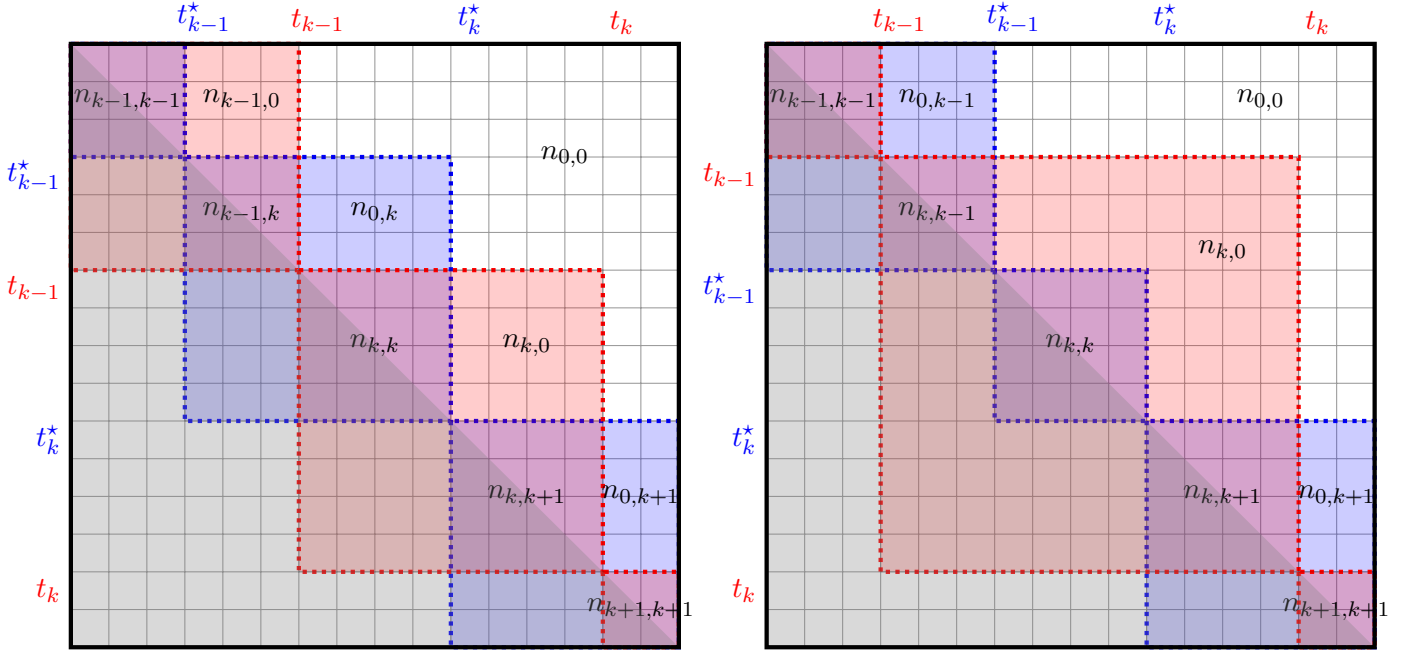
\begin{figure}[!h]
\begin{center}
\begin{tabular}{cc}
\hspace{-15mm}\begin{tikzpicture}
\draw[opacity=0.4] (0,0) grid[step=0.5] (8,8);

\draw (1.75,8) node[above,color=blue]{$\ts_{\kk-1}$};
\draw (0,6.25) node[left,color=blue]{$\ts_{\kk-1}$};
\draw (5.25,8) node[above,color=blue]{$\tsk$};
\draw (0,2.75) node[left,color=blue]{$\tsk$};

\draw[ultra thick,color=blue,dotted] (0,8) rectangle (1.5,6.5);
\draw[fill=blue,opacity=0.2] (0,8) rectangle (1.5,6.5);

\draw[ultra thick,color=blue,dotted] (1.5,6.5) rectangle (5,3);
\draw[fill=blue,opacity=0.2] (1.5,6.5) rectangle (5,3);

\draw[ultra thick,color=blue,dotted] (5,3) rectangle (8,0);
\draw[fill=blue,opacity=0.2] (5,3) rectangle (8,0);

\draw (3.25,8) node[above,color=red]{$\ttt_{\kk-1}$};
\draw (0,4.75) node[left,color=red]{$\ttt_{\kk-1}$};

\draw (7.25,8) node[above,color=red]{$\tk$};
\draw (0,0.75) node[left,color=red]{$\tk$};

\draw[ultra thick,color=red,dotted] (0,8) rectangle (3,5);
\draw[fill=red,opacity=0.2] (0,8) rectangle (3,5);

\draw[ultra thick,color=red,dotted] (3,5) rectangle (7,1);
\draw[fill=red,opacity=0.2] (3,5) rectangle (7,1);

\draw[ultra thick,color=red,dotted] (7,1) rectangle (8,0);
\draw[fill=red,opacity=0.2] (7,1) rectangle (8,0);

\draw (0.75,7.25) node{$\n_{\kk-1,\kk-1}$};
\draw (2.25,5.75) node{$\n_{\kk-1,\kk}$};
\draw (4,4) node{$\n_{\kk,\kk}$};
\draw (6,2) node{$\n_{\kk,\kk+1}$};
\draw (7.5,0.5) node{$\n_{\kk+1,\kk+1}$};

\draw (2.25,7.25) node{$\n_{\kk-1,0}$};
\draw (4,5.75) node{$\n_{0,\kk}$};
\draw (6,4) node{$\n_{\kk,0}$};
\draw (7.5,2) node{$\n_{0,\kk+1}$};

\draw (6.5,6.5) node{$\n_{0,0}$};

\fill[color=gray,opacity=0.3] (0,0) -- (0,8) -- (8,0) -- cycle;
\draw[ultra thick] (0,0) rectangle (8,8);
\end{tikzpicture}&
\hspace{-5mm}\begin{tikzpicture}
\draw[opacity=0.4] (0,0) grid[step=0.5] (8,8);

\draw (3.25,8) node[above,color=blue]{$\ts_{\kk-1}$};
\draw (0,4.75) node[left,color=blue]{$\ts_{\kk-1}$};
\draw (5.25,8) node[above,color=blue]{$\tsk$};
\draw (0,2.75) node[left,color=blue]{$\tsk$};

\draw[ultra thick,color=blue,dotted] (0,8) rectangle (3,5);
\draw[fill=blue,opacity=0.2] (0,8) rectangle (3,5);

\draw[ultra thick,color=blue,dotted] (3,5) rectangle (5,3);
\draw[fill=blue,opacity=0.2] (3,5) rectangle (5,3);

\draw[ultra thick,color=blue,dotted] (5,3) rectangle (8,0);
\draw[fill=blue,opacity=0.2] (5,3) rectangle (8,0);

\draw (1.75,8) node[above,color=red]{$\ttt_{\kk-1}$};
\draw (0,6.25) node[left,color=red]{$\ttt_{\kk-1}$};

\draw (7.25,8) node[above,color=red]{$\tk$};
\draw (0,0.75) node[left,color=red]{$\tk$};

\draw[ultra thick,color=red,dotted] (0,8) rectangle (1.5,6.5);
\draw[fill=red,opacity=0.2] (0,8) rectangle (1.5,6.5);

\draw[ultra thick,color=red,dotted] (1.5,6.5) rectangle (7,1);
\draw[fill=red,opacity=0.2] (1.5,6.5) rectangle (7,1);

\draw[ultra thick,color=red,dotted] (7,1) rectangle (8,0);
\draw[fill=red,opacity=0.2] (7,1) rectangle (8,0);

\draw (0.75,7.25) node{$\n_{\kk-1,\kk-1}$};
\draw (2.25,5.75) node{$\n_{\kk,\kk-1}$};
\draw (4,4) node{$\n_{\kk,\kk}$};
\draw (6,2) node{$\n_{\kk,\kk+1}$};
\draw (7.5,0.5) node{$\n_{\kk+1,\kk+1}$};

\draw (2.25,7.25) node{$\n_{0,\kk-1}$};
\draw (6,5.25) node{$\n_{\kk,0}$};
\draw (7.5,2) node{$\n_{0,\kk+1}$};

\draw (6.5,7.25) node{$\n_{0,0}$};

\fill[color=gray,opacity=0.3] (0,0) -- (0,8) -- (8,0) -- cycle;
\draw[ultra thick] (0,0) rectangle (8,8);
\end{tikzpicture}\\
  \end{tabular}
\end{center}
\caption{ $K=\Ks$ and $\left\|\bt-\bts\right\|_{\infty} < \frac{\n\Deltataus}{2}$.
Left: $\ts_{\kk-1}<\ttt_{\kk-1}<\tsk<\tk$. Right: $\ttt_{\kk-1}<\ts_{\kk-1}<\tsk<\tk$.}
\label{fig:minorationKegalK}
\end{figure}

 We shall assume that $t_k-t_k^{\star} = \left\|\bt-\bts\right\|_{\infty} > 0$.\\

There are two possible configurations (see Figure~\ref{fig:minorationKegalK}). If $t_{k-1}^\star < t_{k-1} < t_k^\star <t_k$, then, by definition of $\Deltataus$, we obtain
\begin{equation}\label{eq:nkk1}
n_{k,k} = \frac{(t_k^\star-t_{k-1})(t_k^\star-t_{k-1}+1)}{2} \geq \frac{\left(\overbrace{(t_k^{\star}-t_{k-1}^\star)}^{\geq n \Deltataus}-\overbrace{(t_{k-1}-t_{k-1}^\star)}^{\leq \left\|\bt-\bts\right\|_{\infty}}\right)(t_k^\star-t_{k-1}+1)}{2}\geq \frac{(n\Deltataus)^2}{8}.
\end{equation}
Otherwise, if $t_{k-1} < t_{k-1}^\star < t_k^\star <t_k$, we obtain
\begin{equation}\label{eq:nkk2}
n_{k,k} = \frac{(t_k^\star-t_{k-1}^\star)(t_k^\star-t_{k-1}^\star+1)}{2} \geq \frac{(n\Deltataus)^2}{2} \geq \frac{(n\Deltataus)^2}{8}.
\end{equation}
Then, by using the above decomposition of $(t_k^\star-t_{k-1})$, we obtain
\begin{eqnarray} \label{eq:n0k}
n_{k,0} & \geq & (t_k^{\star}-t_{k-1})(t_k-t_k^{\star}),  \nonumber\\
& \geq & \left(\underbrace{(t_k^{\star}-t_{k-1}^\star)}_{\geq n \Deltataus}-\underbrace{(t_{k-1}-t_{k-1}^\star)}_{\leq \left\|\bt-\bts\right\|_{\infty}}\right)\underbrace{(t_k-t_k^{\star})}_{=\left\|\bt-\bts\right\|_{\infty}} \geq \frac{\Deltataus}{2} n^2\delta.
\end{eqnarray}

By choosing $(\el=k,\el^\prime=k)$ in \eqref{eq:Bborneinf}, and by using \eqref{eq:nkk1}, \eqref{eq:nkk2} and \eqref{eq:n0k}, we obtain
\begin{equation*}
B_n(t)\geq \frac{1}{n(n+1)}\frac{1}{n_{k}}\frac{(n\Deltataus)^2}{8} \frac{\Deltataus}{2}n^2\delta\lambdainfz^2 \geq \frac{(\Deltataus)^3}{32}\delta \lambdainfz^2.
\end{equation*}

{\it ii)} $\left\|\bt-\bts\right\|_{\infty} \geq \frac{\n\Deltataus}{2}$.\\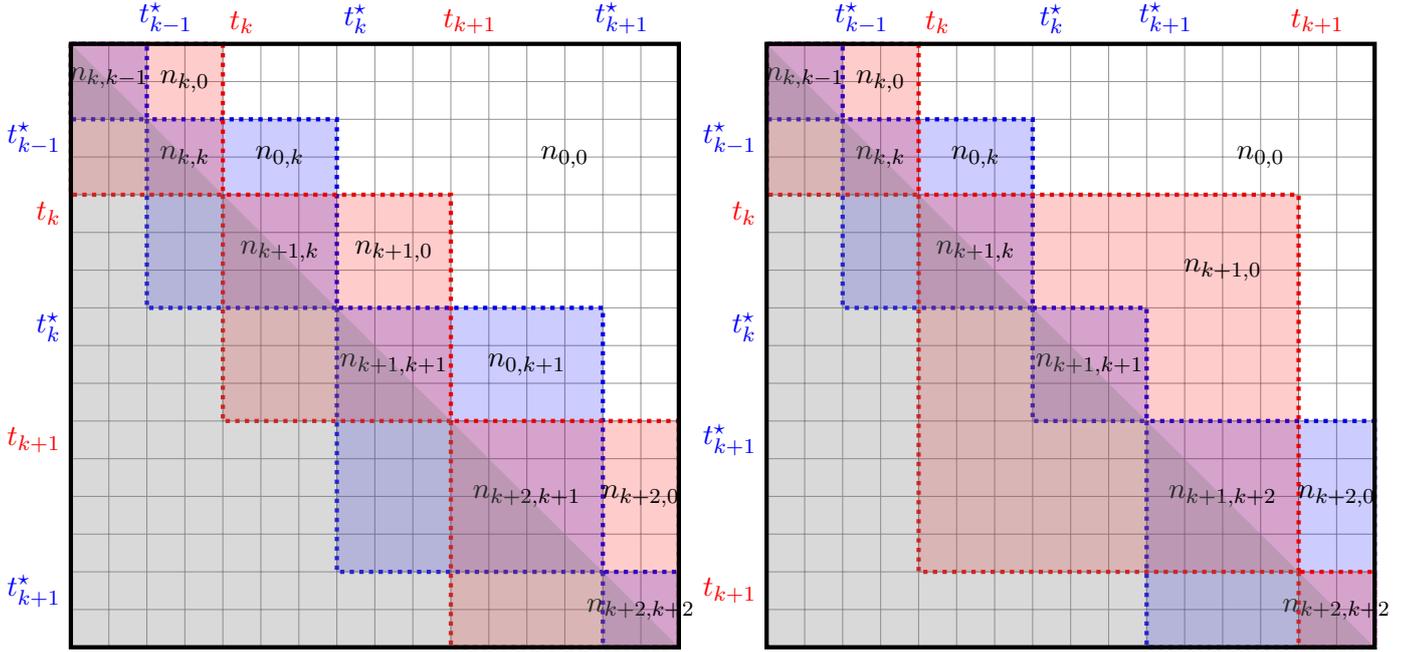
\begin{figure}[!h]
\begin{center}
\begin{tabular}{cc}
\hspace{-15mm}\begin{tikzpicture}
\draw[opacity=0.4] (0,0) grid[step=0.5] (8,8);

\draw (1.25,8) node[above,color=blue]{$\ts_{\kk-1}$};
\draw (0,6.75) node[left,color=blue]{$\ts_{\kk-1}$};
\draw (3.75,8) node[above,color=blue]{$\tsk$};
\draw (0,4.25) node[left,color=blue]{$\tsk$};
\draw (7.25,8) node[above,color=blue]{$\ts_{\kk+1}$};
\draw (0,0.75) node[left,color=blue]{$\ts_{\kk+1}$};

\draw[ultra thick,color=blue,dotted] (0,8) rectangle (1,7);
\draw[fill=blue,opacity=0.2] (0,8) rectangle (1,7);

\draw[ultra thick,color=blue,dotted] (1,7) rectangle (3.5,4.5);
\draw[fill=blue,opacity=0.2] (1,7) rectangle (3.5,4.5);

\draw[ultra thick,color=blue,dotted] (3.5,4.5) rectangle (7,1);
\draw[fill=blue,opacity=0.2] (3.5,4.5) rectangle (7,1);

\draw[ultra thick,color=blue,dotted] (7,1) rectangle (8,0);
\draw[fill=blue,opacity=0.2] (7,1) rectangle (8,0);

\draw (2.25,8) node[above,color=red]{$\ttt_{\kk}$};
\draw (0,5.75) node[left,color=red]{$\ttt_{\kk}$};

\draw (5.25,8) node[above,color=red]{$\ttt_{\kk+1}$};
\draw (0,2.75) node[left,color=red]{$\ttt_{\kk+1}$};

\draw[ultra thick,color=red,dotted] (0,8) rectangle (2,6);
\draw[fill=red,opacity=0.2] (0,8) rectangle (2,6);

\draw[ultra thick,color=red,dotted] (2,6) rectangle (5,3);
\draw[fill=red,opacity=0.2] (2,6) rectangle (5,3);

\draw[ultra thick,color=red,dotted] (5,3) rectangle (8,0);
\draw[fill=red,opacity=0.2] (5,3) rectangle (8,0);

\draw (0.5,7.55) node{$\n_{\kk,\kk-1}$};
\draw (1.5,6.5) node{$\n_{\kk,\kk}$};
\draw (2.75,5.25) node{$\n_{\kk+1,\kk}$};
\draw (4.25,3.75) node{$\n_{\kk+1,\kk+1}$};
\draw (6,2) node{$\n_{\kk+2,\kk+1}$};
\draw (7.5,0.5) node{$\n_{\kk+2,\kk+2}$};

\draw (1.5,7.5) node{$\n_{\kk,0}$};
\draw (2.75,6.5) node{$\n_{0,\kk}$};
\draw (4.25,5.25) node{$\n_{\kk+1,0}$};
\draw (6,3.75) node{$\n_{0,\kk+1}$};
\draw (7.5,2) node{$\n_{\kk+2,0}$};

\draw (6.5,6.5) node{$\n_{0,0}$};

\fill[color=gray,opacity=0.3] (0,0) -- (0,8) -- (8,0) -- cycle;
\draw[ultra thick] (0,0) rectangle (8,8);
\end{tikzpicture}&
\hspace{-5mm}\begin{tikzpicture}
\draw[opacity=0.4] (0,0) grid[step=0.5] (8,8);

\draw (1.25,8) node[above,color=blue]{$\ts_{\kk-1}$};
\draw (0,6.75) node[left,color=blue]{$\ts_{\kk-1}$};
\draw (3.75,8) node[above,color=blue]{$\tsk$};
\draw (0,4.25) node[left,color=blue]{$\tsk$};
\draw (5.25,8) node[above,color=blue]{$\ts_{\kk+1}$};
\draw (0,2.75) node[left,color=blue]{$\ts_{\kk+1}$};

\draw[ultra thick,color=blue,dotted] (0,8) rectangle (1,7);
\draw[fill=blue,opacity=0.2] (0,8) rectangle (1,7);

\draw[ultra thick,color=blue,dotted] (1,7) rectangle (3.5,4.5);
\draw[fill=blue,opacity=0.2] (1,7) rectangle (3.5,4.5);

\draw[ultra thick,color=blue,dotted] (3.5,4.5) rectangle (5,3);
\draw[fill=blue,opacity=0.2] (3.5,4.5) rectangle (5,3);

\draw[ultra thick,color=blue,dotted] (5,3) rectangle (8,0);
\draw[fill=blue,opacity=0.2] (5,3) rectangle (8,0);

\draw (2.25,8) node[above,color=red]{$\ttt_{\kk}$};
\draw (0,5.75) node[left,color=red]{$\ttt_{\kk}$};

\draw (7.25,8) node[above,color=red]{$\ttt_{\kk+1}$};
\draw (0,0.75) node[left,color=red]{$\ttt_{\kk+1}$};

\draw[ultra thick,color=red,dotted] (0,8) rectangle (2,6);
\draw[fill=red,opacity=0.2] (0,8) rectangle (2,6);

\draw[ultra thick,color=red,dotted] (2,6) rectangle (7,1);
\draw[fill=red,opacity=0.2] (2,6) rectangle (7,1);

\draw[ultra thick,color=red,dotted] (7,1) rectangle (8,0);
\draw[fill=red,opacity=0.2] (7,1) rectangle (8,0);

\draw (0.5,7.55) node{$\n_{\kk,\kk-1}$};
\draw (1.5,6.5) node{$\n_{\kk,\kk}$};
\draw (2.75,5.25) node{$\n_{\kk+1,\kk}$};
\draw (4.25,3.75) node{$\n_{\kk+1,\kk+1}$};
\draw (6,2) node{$\n_{\kk+1,\kk+2}$};
\draw (7.5,0.5) node{$\n_{\kk+2,\kk+2}$};

\draw (1.5,7.5) node{$\n_{\kk,0}$};
\draw (2.75,6.5) node{$\n_{0,\kk}$};
\draw (6,5) node{$\n_{\kk+1,0}$};
\draw (7.5,2) node{$\n_{\kk+2,0}$};

\draw (6.5,6.5) node{$\n_{0,0}$};

\fill[color=gray,opacity=0.3] (0,0) -- (0,8) -- (8,0) -- cycle;
\draw[ultra thick] (0,0) rectangle (8,8);
\end{tikzpicture}\\
  \end{tabular}
\end{center}
\caption{ $K=\Ks$ and $\left\|\bt-\bts\right\|_{\infty} \geq \frac{\n\Deltataus}{2}$.
Left: $t_{k} < t_{k}^\star < t_{k+1}^\star <t_{k+1}$. Right: $t_{k} < t_{k}^\star < t_{k+1} <t_{k+1}^\star$.}
\label{fig:minorationKegalKbis}
\end{figure}

Since $K=K^{\vrai}$, there exists $k$ such that $t_{k}^{\star}-t_{k}\geq n \frac{\Deltataus}{2}$ and
$t_{k+1}-t_{k}^{\star}\geq n \frac{\Deltataus}{2}$ (otherwise, this would imply that $\K>\Ks$).
As above, there are two possible cases, either $t_{k} < t_{k}^\star < t_{k+1}^\star <t_{k+1}$ or
$t_{k} < t_{k}^\star < t_{k+1} <t_{k+1}^\star$ (see Figure~\ref{fig:minorationKegalKbis}).

If $t_{k} < t_{k}^\star < t_{k+1}^\star <t_{k+1}$, we obtain, by definition of $\Deltataus$,
\begin{equation}\label{eq:nk1k11}
n_{k+1,k+1}=\frac{(t_{k+1}^{\star}-t_{k}^{\star})(t_{k+1}^{\star}-t_{k}^{\star}+1)}{2}\geq \frac{(n\Deltataus)^2}{2},
\end{equation}
and
\begin{equation}\label{eq:nk101}
n_{k+1,0} \geq (t_{k+1}^{\star}-t_{k}^{\star})(t_k^{\star}-t_k) \geq (n\Deltataus)\frac{(n\Deltataus)}{2}.
\end{equation}

If $t_{k} < t_{k}^\star < t_{k+1} <t_{k+1}^\star$, we obtain
\begin{equation}\label{eq:nk1k12}
n_{k+1,k+1}=\frac{(t_{k+1}-t_{k}^{\star})(t_{k+1}-t_{k}^{\star}+1)}{2}\geq \frac{1}{2}\left(\frac{n\Deltataus}{2}\right)^2,
\end{equation}
and
\begin{equation}\label{eq:nk102}
n_{k+1,0} \geq (t_{k+1}-t_{k}^{\star})(t_k^{\star}-t_k) \geq \left(\frac{(n\Deltataus)}{2}\right)^2.
\end{equation}

By choosing $(\el=\el^\prime=k+1)$ in \eqref{eq:Bborneinf}, and by using \eqref{eq:nk1k11}, \eqref{eq:nk101}, \eqref{eq:nk1k12} and \eqref{eq:nk102}, we obtain

\begin{equation*}
B_n(t) \geq \frac{(\Deltataus)^4}{32} \lambdainfz^2.
\end{equation*}

\subsection{Deviation inequalities}
\begin{lemma}\label{lem:maj:Vn}
For all $\alpha>0$,
\begin{eqnarray*}
\Prob\left(-\underset{\bt\in\AnK^{1/\n}}{\min}\Vn(\bt) \geq \alpha\right) &\leq& n(n+1)e^{-\frac{n(n+1)\alpha}{16K\beta}} + 2e^{-\frac{|\Gun|\alpha}{8\beta}},
\end{eqnarray*}
where $\Vn$ is defined by (\ref{eq:Kn_Vn_Wn}) and (\ref{eq:VD_Vz}) and $\AnK^{1/\n}$ is defined in (\ref{eq:AnK}) with $\Delta_n=1/n$.
Moreover, if $\alpha=\alpha_n$ is such that $\alpha_n n^2/\log(n)\to\infty$ and $\alpha_n |G_{0,1}|\to\infty$,
as $n$ tends to infinity, then
$$
\Prob\left(-\underset{\bt\in\AnK^{1/\n}}{\min}\Vn(\bt) \geq \alpha_n\right)\to 0, \textrm { as } n\to+\infty.
$$
\end{lemma}
\textcolor{black}{The proof is given in the supplementary material.}

\begin{lemma}\label{lem:maj:Wn}
Let $\Wn$ be defined by (\ref{eq:Kn_Vn_Wn}) and (\ref{eq:WD_Wz}), then there exists $C_1>0$ such that for all $\alpha>0$ :
\begin{eqnarray*}
\Prob\left(-\underset{\bt\in\AnK^{1/\n}}{\min}\Wn(\bt)\geq\alpha\right)
&\leq&C_1\n^{4\Kmax}\exp\left[-\frac{\alpha^2\n(\n+1)}{128\, \betaa \left(\K+1\right)^2\left(\Ks+1\right)^2{\lambdasup}^2}\right],
\end{eqnarray*}
where $\lambdasup=\underset{\kk\neq\el}{\sup}\left|\muks-\muls\right|$ and $\AnK^{1/\n}$
is defined in (\ref{eq:AnK}) with $\Delta_n=1/n$.
Moreover, if $\alpha=\alpha_n$ is such that $\alpha_n^2 n^2/\log(n)\to\infty$,
as $n$ tends to infinity, then
$$
\Prob\left(-\underset{\bt\in\AnK^{1/\n}}{\min}\Wn(\bt) \geq \alpha_n\right)\to 0, \textrm { as } n\to\infty.
$$
\end{lemma}

\textcolor{black}{The proof is given in the supplementary material.}

\begin{lemma}\label{Lemme:InegaliteConcentrationZn}
For all $\alpha>0$ and $\gammaa>0$,
\begin{eqnarray*}
\Prob\left(- \underset{\bt\in\AnK^{1/\n}}{\min}\Zn(\bt)\geq\alpha\right)\leq 2e^{-\frac{|\Gun|\gamma^2}{4\beta}}
+2C_1\n^{4\Kmax} e^{-\frac{\alpha^2\n(\n+1)}{512\gamma^2\betaa}}
+ 2e^{-\frac{|\Gun|\alpha^2n^2}{32\bar{\lambda}^2\beta} },
\end{eqnarray*}
where $\Zn$ is defined by (\ref{eq:Z}), $\AnK^{1/\n}$ is defined in (\ref{eq:AnK}) with $\Delta_n=1/n$ and
$\lambdasup=\underset{\kk\neq\el}{\sup}\left|\muks-\muls\right|$.
Moreover, if $\alpha=\alpha_n$ is such that $\alpha_n^2 n^2/\log(n)\to\infty$ and $\alpha_n^2 n^2 |G_{0,1}|\to\infty$,
as $n$ tends to infinity, then
$$
\Prob\left(-\underset{\bt\in\AnK^{1/\n}}{\min}\Zn(\bt) \geq \alpha_n\right)\to 0, \textrm { as } n\to\infty.
$$
\end{lemma}

\textcolor{black}{The proof is given in the supplementary material.}

\bibliographystyle{abbrvnat}
\bibliography{biblio}

\pagebreak
\section*{Supplementary material}
\subsection*{Proof of Equation~(\ref{eq:dec:Jn})}\label{Annexe:proof:Jn}
The goal of this section is to prove that
$\Jn$ defined in (\ref{eq:def_Jn}) can be rewritten as in (\ref{eq:dec:Jn}).
By definition of $\Q_n^{\K}$ given in (\ref{def:QnK}),
\begin{equation}\label{eq:Jn}
\Jn(\bt)
=\J_n^D(\bt)+\J_n^0(\bt),
\end{equation}
where
$$
\J_n^D(\bt)=\frac{2}{\n(\n+1)}\left(\left[\sum_{\kk=1}^{\K}{\sum_{(\ii,\jj)\in\Dk}{\left(\Yij-\YbarDk\right)^2}}\right] -\left[\sum_{\kk=1}^{\K^\star}{\sum_{(\ii,\jj)\in\Dsk}{\left(\Yij-\YbarDsk\right)^2}}\right]\right),
$$
and
$$
\J_n^0(\bt)=\frac{2}{\n(\n+1)}\left(\sum_{(\ii,\jj)\in\Ez}{\left(\Yij-\YbarGun\right)^2} -\sum_{(\ii,\jj)\in\Ezs}{\left(\Yij-\YbarGun\right)^2}\right).
$$
Using (\ref{eq:model}) and $\YbarDk=\mbE\left[\YbarDk\right]+|\Dk|^{-1}\sum_{(\ip,\jp)\in\Dk}{\epsipjp}$, where $|A|$
denotes the cardinality of the set $A$, we obtain
\begin{eqnarray*}
&&\sum_{(\ii,\jj)\in\Dk}{\left(\Yij-\YbarDk\right)^2}
=\sum_{(\ii,\jj)\in\Dk}{\left[\Yij^2 -2\Yij\YbarDk+\YbarDk^2\right]}\\
&=&\sum_{(\ii,\jj)\in\Dk}{\left[\mbE\left[\Yij\right]^2+2\mbE\left[\Yij\right]\epsij+\epsij^2\right]}\\
        &-2&\sum_{(\ii,\jj)\in\Dk}\left[\mbE\left[\Yij\right]\mbE\left[\YbarDk\right]+\epsij\mbE\left[\YbarDk\right]
        +\mbE\left[\Yij\right]\frac{1}{|\Dk|}\left(\sum_{(\ip,\jp)\in\Dk}{\epsipjp}\right)+\epsij\frac{1}{|\Dk|}\left(\sum_{(\ip,\jp)\in\Dk}{\epsipjp}\right)\right]\\
        &+&\sum_{(\ii,\jj)\in\Dk}\left[\mbE\left[\YbarDk\right]^2 +2\mbE\left[\YbarDk\right]\frac{1}{|\Dk|}\left(\sum_{(\ip,\jp)\in\Dk}{\epsipjp}\right) +\frac{1}{|\Dk|^2}\left(\sum_{(\ip,\jp)\in\Dk}{\epsipjp}\right)^2\right].
\end{eqnarray*}
By gathering the deterministic terms and the terms linked to the noise, we obtain
\begin{eqnarray}\label{eq:Dk}
&&\sum_{(\ii,\jj)\in\Dk}{\left(\Yij-\YbarDk\right)^2}
=\sum_{(\ii,\jj)\in\Dk}{\left(\mbE\left[\Yij\right]-\mbE\left[\YbarDk\right]\right)^2}
-2\frac{1}{|\Dk|}\left(\sum_{(\ip,\jp)\in\Dk}{\epsipjp}\right)\sum_{(\ii,\jj)\in\Dk}{\mbE\left[\Yij\right]}\nonumber\\
&-&\frac{1}{|\Dk|}\left(\sum_{(\ip,\jp)\in\Dk}{\epsipjp}\right)^2
+\sum_{(\ii,\jj)\in\Dk}\left(2\mbE\left[\Yij\right]\epsij+\epsij^2\right).
\end{eqnarray}
Thus, for $\Jn^D$ defined in (\ref{eq:Jn}), we obtain
\begin{eqnarray}\label{eq:JnD}
&&\Jn^D(\bt)=
\frac{2}{n(n+1)}\left[\sum_{k=1}^K\sum_{(\ii,\jj)\in\Dk}{\left(\mbE\left[\Yij\right]-\mbE\left[\YbarDk\right]\right)^2}
-\underbrace{\sum_{k=1}^{\Ks} \sum_{(\ii,\jj)\in\Dsk}{\left(\mbE\left[\Yij\right]-\mbE\left[\YbarDsk\right]\right)^2}}_{=0}\right]\nonumber\\
&-&\frac{4}{n(n+1)}\left[\sum_{k=1}^K \frac{1}{|\Dk|}\left(\sum_{(\ip,\jp)\in\Dk}{\epsipjp}\right)\sum_{(\ii,\jj)\in\Dk}{\mbE\left[\Yij\right]}
-\sum_{k=1}^{\Ks} \frac{1}{|\Dsk|}\left(\sum_{(\ip,\jp)\in\Dsk}{\epsipjp}\right)\sum_{(\ii,\jj)\in\Dsk}{\mbE\left[\Yij\right]}\right]
\nonumber\\
&-&\frac{2}{n(n+1)}\left[\sum_{k=1}^K \frac{1}{|\Dk|}\left(\sum_{(\ip,\jp)\in\Dk}{\epsipjp}\right)^2
-\sum_{k=1}^{\Ks} \frac{1}{|\Dsk|}\left(\sum_{(\ip,\jp)\in\Dsk}{\epsipjp}\right)^2\right]\nonumber\\
&+&\frac{2}{\n(\n+1)}\left(\sum_{k=1}^K\sum_{(\ii,\jj)\in\Dk}\left(2\mbE\left[\Yij\right]\epsij+\epsij^2\right)
-\sum_{k=1}^{\Ks}\sum_{(\ii,\jj)\in\Dsk}\left(2\mbE\left[\Yij\right]\epsij+\epsij^2\right)\right)\nonumber\\
&=&\KDn(\bt)+\WDn(\bt)+\VDn(\bt)\nonumber\\
&+&\frac{2}{\n(\n+1)}\left(\sum_{k=1}^K\sum_{(\ii,\jj)\in\Dk}\left(2\mbE\left[\Yij\right]\epsij+\epsij^2\right)
-\sum_{k=1}^{\Ks}\sum_{(\ii,\jj)\in\Dsk}\left(2\mbE\left[\Yij\right]\epsij+\epsij^2\right)\right),
\end{eqnarray}
since $\mbE[\Yij]=\mbE[\YbarDsk]$, for all $(\ii,\jj)\in\Dsk$.
Using (\ref{eq:G00:G01}), we obtain
\begin{equation*}
\Jzn(\bt)=\frac{2}{\n(\n+1)}\left(\sum_{(\ii,\jj)\in\Gzero}{\left(\Yij-\YbarGun\right)^2} -\sum_{(\ii,\jj)\in\Gszero}{\left(\Yij-\YbarGun\right)^2}\right).
\end{equation*}
Using (\ref{eq:model}) and $\YbarGun=\mbE\left[\YbarGun\right]+|\Gun|^{-1}\sum_{(\ip,\jp)\in\Gun}{\epsipjp}$, we obtain
\begin{eqnarray*}
&&\sum_{(\ii,\jj)\in\Gzero}{\left(\Yij-\YbarGun\right)^2}=\sum_{(\ii,\jj)\in\Gzero}{\left[\Yij^2 -2\Yij\YbarGun+\YbarGun^2\right]}\\
&=&\sum_{(\ii,\jj)\in\Gzero}{\left[\mbE\left[\Yij\right]^2+2\mbE\left[\Yij\right]\epsij+\epsij^2\right]}\\
&-&2\sum_{(\ii,\jj)\in\Gzero}\left\{\mbE\left[\Yij\right]\mbE\left[\YbarGun\right]+\epsij\mbE\left[\YbarGun\right]
+\mbE\left[\Yij\right]\frac{1}{|\Gun|}\left(\sum_{(\ip,\jp)\in\Gun}{\epsipjp}\right)\right.\\
&&\left.+\epsij\frac{1}{|\Gun|}\left(\sum_{(\ip,\jp)\in\Gun}{\epsipjp}\right)\right\}\\
&+&\sum_{(\ii,\jj)\in\Gzero}\left\{\mbE\left[\YbarGun\right]^2 +2\mbE\left[\YbarGun\right]\frac{1}{|\Gun|}\left(\sum_{(\ip,\jp)\in\Gun}{\epsipjp}\right) +\frac{1}{|\Gun|^2}\left(\sum_{(\ip,\jp)\in\Gun}{\epsipjp}\right)^2\right\}.
\end{eqnarray*}
By gathering the deterministic terms and the terms linked to the noise as in (\ref{eq:Dk}), we obtain
\begin{eqnarray*}
&&\sum_{(\ii,\jj)\in\Gzero}{\left(\Yij-\YbarGun\right)^2}
=\sum_{(\ii,\jj)\in\Gzero}{\left(\mbE\left[\Yij\right]-\mbE\left[\YbarGun\right]\right)^2}\\
&&-2\muzs\sum_{(\ii,\jj)\in\Gzero}{\epsij}
-2\frac{|\Gzero|}{|\Gun|}\left(\sum_{(\ip,\jp)\in\Gun}{\epsipjp}\right)\left[\frac{1}{|\Gzero|}
\sum_{(\ii,\jj)\in\Gzero}\mbE(\Yij)-\muzs\right]\\
&+&\frac{1}{|\Gun|}\left(\sum_{(\ip,\jp)\in\Gun}{\epsipjp}\right)\left[\frac{|\Gzero|}{|\Gun|}\left(\sum_{(\ip,\jp)\in\Gun}{\epsipjp}\right) -2\left(\sum_{(\ii,\jj)\in\Gzero}{\epsij}\right)\right]\\
&+&\sum_{(\ii,\jj)\in\Gzero}{\left(2\mbE\left[\Yij\right]\epsij+\epsij^2\right)},
\end{eqnarray*}
where $\muzs$ is defined in (\ref{eq:bloc_mu}).
Thus, we obtain
\begin{eqnarray}\label{eq:Jzn}
&&\Jzn(\bt)=\frac{2}{\n(\n+1)}
\left(\sum_{(\ii,\jj)\in\Gzero}{\left(\mbE\left[\Yij\right]-\mbE\left[\YbarGun\right]\right)^2}
-\sum_{(\ii,\jj)\in\Gszero}{\left(\mbE\left[\Yij\right]-\mbE\left[\YbarGun\right]\right)^2}\right)\nonumber\\
&&-\frac{4}{\n(\n+1)}\muzs\left(\sum_{(\ii,\jj)\in\Gzero}{\epsij}-\sum_{(\ii,\jj)\in\Gszero}{\epsij}\right)\nonumber\\
&&+\frac{2}{\n(\n+1)}\frac{1}{|\Gun|^2}\left(\sum_{(\ip,\jp)\in\Gun}{\epsipjp}\right)^2 \left(|\Gzero|-|\Gszero|\right)
\nonumber\\
&&-\frac{4}{\n(\n+1)}\frac{1}{|\Gun|}\left(\sum_{(\ip,\jp)\in\Gun}{\epsipjp}\right)
\left[\sum_{(\ii,\jj)\in\Gzero}\mbE(\Yij)-\sum_{(\ii,\jj)\in\Gszero}\mbE(\Yij)-\muzs\left(|\Gzero|-|\Gszero|\right)\right]
\nonumber\\
&&-\frac{4}{\n(\n+1)}\frac{1}{|\Gun|}\left(\sum_{(\ip,\jp)\in\Gun}{\epsipjp}\right)
\left(\sum_{(\ii,\jj)\in\Gzero}{\epsij}-\sum_{(\ii,\jj)\in\Gszero}{\epsij}\right)
\nonumber\\
&&+\frac{2}{\n(\n+1)}\left(\sum_{(\ii,\jj)\in\Gzero}{\left(2\mbE\left[\Yij\right]\epsij+\epsij^2\right)} -\sum_{(\ii,\jj)\in\Gszero}{\left(2\mbE\left[\Yij\right]\epsij+\epsij^2\right)}\right)
\nonumber\\
&=& \Kzn(\bt)+\Wzn(\bt)+\Vzn(\bt)+\Zn(\bt) \nonumber\\
&&+\frac{2}{\n(\n+1)}\left(\sum_{(\ii,\jj)\in\Gzero}{\left(2\mbE\left[\Yij\right]\epsij+\epsij^2\right)}-\sum_{(\ii,\jj)\in\Gszero}{\left(2\mbE\left[\Yij\right]\epsij+\epsij^2\right)}\right),
\end{eqnarray}
since $\mbE[\Yij]=\muzs=\mbE\left[\YbarGun\right]$, for all $(\ii,\jj)\in\Gszero$. 
Then, from (\ref{eq:Jn}), (\ref{eq:JnD}) and (\ref{eq:Jzn}), we obtain
\begin{eqnarray*}
&&\Jn(\bt)=\KDn(\bt)+\WDn(\bt)+\VDn(\bt)+\Kzn(\bt)+\Wzn(\bt)+\Vzn(\bt)+\Zn(\bt)\\
&+&\frac{2}{\n(\n+1)}\left(\sum_{k=1}^K\sum_{(\ii,\jj)\in\Dk}\left(2\mbE\left[\Yij\right]\epsij+\epsij^2\right)
-\sum_{k=1}^{\Ks}\sum_{(\ii,\jj)\in\Dsk}\left(2\mbE\left[\Yij\right]\epsij+\epsij^2\right)\right)\\
&+&\frac{2}{\n(\n+1)}\left(\sum_{(\ii,\jj)\in\Gzero}{\left(2\mbE\left[\Yij\right]\epsij+\epsij^2\right)}-\sum_{(\ii,\jj)\in\Gszero}{\left(2\mbE\left[\Yij\right]\epsij+\epsij^2\right)}\right).
\end{eqnarray*}
Note that
\begin{eqnarray*}
&&\sum_{k=1}^K \sum_{(\ii,\jj)\in\Dk}\left(2\mbE\left[\Yij\right]\epsij+\epsij^2\right)
+\sum_{(\ii,\jj)\in\Gzero}{\left(2\mbE\left[\Yij\right]\epsij+\epsij^2\right)}
+\sum_{(\ii,\jj)\in\Gun}{\left(2\mbE\left[\Yij\right]\epsij+\epsij^2\right)}\\
&=&\sum_{k=1}^{\Ks}\sum_{(\ii,\jj)\in\Dsk}\left(2\mbE\left[\Yij\right]\epsij+\epsij^2\right)
+\sum_{(\ii,\jj)\in\Gszero}{\left(2\mbE\left[\Yij\right]\epsij+\epsij^2\right)}
+\sum_{(\ii,\jj)\in\Gun}{\left(2\mbE\left[\Yij\right]\epsij+\epsij^2\right)},
\end{eqnarray*}
since it amounts to summing up over all the possible indices $i$ and $j$. This concludes the proof.


\subsection*{Proof of Lemma \ref{lem:maj:Vn}}

By (\ref{eq:Kn_Vn_Wn}) and (\ref{eq:VD_Vz}),
\begin{eqnarray*}
\Vn(\bt) &=& \ga{ \frac{2}{\n(\n+1)}\left[\sum_{\kk=1}^{\K^*}
\frac{\left(\sum_{(\ii,\jj)\in\Dsk}{\epsij}\right)^2}{\left|\Dsk\right|}
-\sum_{\kp=1}^{\K}\frac{\left(\sum_{(\ip,\jp)\in\Dkp}{\epsipjp}\right)^2}{\left|\Dkp\right|}\right] } \\
&& + \ga{ \frac{2}{\n(\n+1)}\frac{1}{|\Gun|^2}\left(\sum_{(\ii,\jj)\in\Gun}{\epsij}\right)^2\left(|\Gzero|-|\Gszero|\right) }\\
&=& \VDn(\bt) + \Vzn(\bt).
\end{eqnarray*}
Hence,
\begin{eqnarray}\label{eq:prob:minVn}
\!\!\!\Prob\left(-\underset{\bt\in\AnK^{1/n}}{\min}\Vn(\bt) \geq \alpha\right) &\!\!\!\leq& \Prob\left(-\underset{\bt\in\AnK^{1/n}}{\min}\VDn(\bt) \geq \frac{\alpha}{2}\right) + \Prob\left(-\underset{\bt\in\AnK^{1/n}}{\min}\Vzn(\bt) \geq \frac{\alpha}{2}\right)\nonumber\\
&\!\!\!=& \Prob\left(\underset{\bt\in\AnK^{1/n}}{\max}-\VDn(\bt) \geq \frac{\alpha}{2}\right)
+ \Prob\left(\underset{\bt\in\AnK^{1/n}}{\max}-\Vzn(\bt) \geq \frac{\alpha}{2}\right).
\end{eqnarray}
Let us now derive an upper bound for each of these two probabilities. Let us first address the first term in the rhs of (\ref{eq:prob:minVn}).
Note that
\begin{eqnarray*}
\VDn(\bt)
&\geq&-\frac{2}{\n(\n+1)}\left[\sum_{\kk=1}^{\K}{\frac{\left(\sum_{(\ii,\jj)\in\Dk}{\epsij}\right)^2}{\left|\Dk\right|}}\right]
\geq-\frac{2}{\n(\n+1)}\left[\sum_{\kk=1}^{\K}{ \underset{\A\in\mcID}{\max}\frac{\left(\sum_{(\ii,\jj)\in\A}{\epsij}\right)^2}{\left|\A\right|}}\right] \\
&\geq&-\frac{2\K}{\n(\n+1)}\left[ \underset{\A\in\mcID}{\max}\frac{\left(\sum_{(\ii,\jj)\in\A}{\epsij}\right)^2}{\left|\A\right|}\right],
\end{eqnarray*}
where $\mcID$ is the set of all possible diagonal blocks:
\begin{equation}\label{eq:EnsembleID}
\mcID=\left\{\left\{\left.(\ii,\jj)\in\{1,\ldots,\n\}^2\right|\;\ttt\leq\ii\leq\jj\leq\tp\right\}\Big|1\leq\ttt<\tp\leq\n\right\}.
\end{equation}
Thus,
\begin{eqnarray*}
\probT{\underset{\bt\in\AnK^{1/n}}{\max}-\VDn(\bt) \geq \frac{\alpha}{2} } &\leq& \probT{\frac{2\K}{\n(\n+1)}\left[ \underset{\A\in\mcID}{\max}\frac{\left(\sum_{(\ii,\jj)\in\A}{\epsij}\right)^2}{\left|\A\right|}\right] \geq \frac{\alpha}{2} }.
\end{eqnarray*}
Moreover, from Lemma \ref{concentration.subgauss.mean} we obtain that for all positive $u$,
\begin{eqnarray*}
\probT{ \underset{\A\in\mcID}{\max}\frac{\left(\sum_{(\ii,\jj)\in\A}{\epsij}\right)^2}{\left|\A\right|} \geq u }
&\leq& \sum_{\A\in\mcID} \probT{\left|\sum_{(\ii,\jj)\in\A}{\epsij}\right|\geq\sqrt{u |A|}}
\leq 2\sum_{\A\in\mcID} e^{-\frac{u}{4\beta}}\\
&\leq& n(n+1) e^{-\frac{u}{4\beta}},
\end{eqnarray*}
where we use $\left|\mcID\right|\leq\n(\n+1)/2$. Setting $u=\frac{n(n+1)\alpha}{4K}$ in the previous inequality yields
\begin{eqnarray}
\probT{\underset{\bt\in\AnK^{1/n}}{\max}-\VDn(\bt) \geq \frac{\alpha}{2} } &\leq& n(n+1)e^{-\frac{n(n+1)\alpha}{16K\beta}} \ .\label{maj1vn}
\end{eqnarray}
Let us now address the second term in the rhs of (\ref{eq:prob:minVn}). Note that
\begin{eqnarray*}
\Vzn(\bt)&=&\frac{2}{\n(\n+1)}\frac{1}{|\Gun|^2}\left(\sum_{(\ii,\jj)\in\Gun}{\epsij}\right)^2\left(|\Gzero|-|\Gszero|\right)
    \geq -\frac{2|\Gszero|}{\n(\n+1)}\frac{1}{|\Gun|^2}\left(\sum_{(\ii,\jj)\in\Gun}{\epsij}\right)^2\\
    &\geq&-\left(\frac{\sum_{(\ii,\jj)\in\Gun}{\epsij}}{|\Gun|}\right)^2\ ,
\end{eqnarray*}
which leads to
\begin{eqnarray}
\probT{\underset{\bt\in\AnK}{\max}-\Vzn(\bt) \geq \frac{\alpha}{2} } 
\leq \probT{\left|\frac{\sum_{(\ii,\jj)\in\Gun}{\epsij}}{|\Gun|}\right| \geq \sqrt{\frac{\alpha}{2}} }
\leq 2e^{-\frac{|\Gun|\alpha}{8\beta}} \label{maj2vn},
\end{eqnarray}
where the last inequality comes from Lemma \ref{concentration.subgauss.mean}.
The conclusion follows from \eqref{maj1vn} and \eqref{maj2vn}.

\subsection*{Proof of Lemma \ref{lem:maj:Wn}}

Using (\ref{eq:YbarDk}) $\Wn$ can be rewritten as follows:
\begin{eqnarray}\label{eq:Wn1}
\Wn(\bt)&=&\frac{4}{\n(\n+1)}\left[\sum_{k'=0}^{\Ks}{\left(\sum_{(\ip,\jp)\in D^\star_{k'}}{\epsipjp}\right)\mu^\star_{k'}}
-\sum_{k=0}^{\K}{\left\{ \left(\sum_{(i,j)\in\Dk}{\epsij}\right)\sum_{\el=0}^{\K^{\vrai}}\frac{\nkl}{\nk}\muls\right\}}\right]\nonumber\\
&&+\frac{4}{\n(\n+1)}\left(\sum_{(i,j)\in D_0}{\epsij}\right)\sum_{\el=0}^{\K^{\vrai}}\frac{n_{0,\el}}{n_0}(\muls-\muzs),
\end{eqnarray}
where we used that $\muzs=\sum_{\el=0}^{\K^{\vrai}}\frac{n_{0,\el}}{n_0}\muzs$. With the notation
$$
e_{k,\el}=\sum_{(\ii,\jj)\in \Dsk\cap\Dl}{\epsij},
$$
the first term in the rhs of the previous equation can be rewritten as follows:
\begin{eqnarray}\label{eq:Wn2}
&&\frac{4}{\n(\n+1)}\left[\sum_{k'=0}^{\Ks}{\left(\sum_{\el'=0}^K\sum_{(\ip,\jp)\in D^\star_{k'}\cap D_{\el'}}{\epsipjp}\right)\mu^\star_{k'}}
-\sum_{k=0}^{\K}{\left[\left(\sum_{k_1=0}^{\Ks}\sum_{(\ii,\jj)\in D^\star_{k_1}\cap\Dk}{\epsij}\right)\sum_{\el=0}^{\K^{\vrai}}\frac{\nkl}{\nk}\muls\right]}
\right]\nonumber\\
&=&\frac{4}{\n(\n+1)}\left[\sum_{k'=0}^{\Ks}{\sum_{\el'=0}^K e_{k',\el'}\sum_{\el_1=0}^{\K^{\vrai}}\frac{n_{\el',\el_1}}{n_{\el'}}\mu^\star_{k'}}
-\sum_{k=0}^{\K}\sum_{k_1=0}^{\Ks} e_{k_1,k}\sum_{\el=0}^{\K^{\vrai}}\frac{\nkl}{\nk}\muls\right]
\end{eqnarray}
where we used that $\sum_{\el_1=0}^{\K^{\vrai}}n_{\el',\el_1}/n_{\el'}\mu^\star_{k'}=\mu^\star_{k'}$.
We deduce from (\ref{eq:Wn1}) and (\ref{eq:Wn2}) that
\begin{eqnarray*}
\Wn(\bt)=\frac{4}{\n(\n+1)}\sum_{\kk=0}^{\K}{ \sum_{\lp=0}^{\Ks}{\sum_{\el=0}^{\Ks}{\elpk\frac{\nkl}{\nk}{\left(\mulps-\muls\right)}}}}
+\frac{4}{\n(\n+1)}\left(\sum_{(i,j)\in D_0}{\epsij}\right)\sum_{\el=0}^{\K^{\vrai}}\frac{n_{0,\el}}{n_0}(\muls-\muzs).
\end{eqnarray*}
Thus,
\begin{eqnarray*}
&&|\Wn(\bt)|\leq\frac{8}{\n(\n+1)}\sum_{\kk=0}^{\K}{ \sum_{\lp=0}^{\Ks}{\sum_{\el=0}^{\Ks}{\left[\underset{\A\in\mcIRK\cup\mcID}{\max} \left|\sum_{(\ip,\jp)\in\A}{\epsipjp}\right|\right] \frac{\nkl}{\nk}{\lambdasup}}}}\\
&\leq&\frac{8}{\n(\n+1)}\lambdasup \left[\underset{\A\in\mcIRK\cup\mcID}{\max}\left|\sum_{(\ip,\jp)\in\A}{\epsipjp}\right|\right]\sum_{\kk=0}^{\K}{ \sum_{\lp=0}^{\Ks}{ \frac{\nk}{\nk}}}\\
&\leq&\frac{8\left(\K+1\right)\left(\Ks+1\right)}{\n(\n+1)}\lambdasup \left[\underset{\A\in\mcIRK\cup\mcID}{\max}\left|\sum_{(\ip,\jp)\in\A}{\epsipjp}\right|\right],
\end{eqnarray*}
where $\mcID$ is defined in (\ref{eq:EnsembleID}) and
\begin{eqnarray}\label{eq:def_IRK}
 \mcIRK&=&\left\{\left.\bigcup_{\kk=1}^{\Kmax}\left\{\left.(\ii,\jj)\in\{1,\ldots,\n\}^2\right|\ii_{\kk}^{(1)}\leq\ii<\ii_{\kk}^{(2)}\text{ and }\jj_{\kk}^{(1)}\leq\jj<\jj_{\kk}^{(2)}\right\} \right|\right.\nonumber\\
     & &\quad\quad\forall\kk,\;1\leq\ii_{\kk}^{(1)}\leq\ii_{\kk}^{(2)}\leq\jj_{\kk}^{(1)}\leq\jj_{\kk}^{(2)}\leq\n+1\Bigg\}.
\end{eqnarray}
Using Lemma \ref{concentration.max.set}, we obtain that
\begin{eqnarray*}
&&\Prob\left(-\underset{\bt\in\AnK^{1/\n}}{\min}\Wn(\bt)\geq\alpha\right)
\leq \Prob\left(\frac{8\left(\K+1\right)\left(\Ks+1\right)}{\n(\n+1)}\lambdasup \left[\underset{\A\in\mcIRK\cup\mcID}{\max}
\left|\sum_{(\ip,\jp)\in\A}{\epsipjp}\right|\right]>\alpha\right)\\
&\leq&\left|\mcIRK\cup\mcID\right| \exp\left[-\frac{\alpha^2\n(\n+1)}{128\, \betaa \left(\K+1\right)^2\left(\Ks+1\right)^2{\lambdasup}^2}\right],
\end{eqnarray*}
which concludes the proof by using that $\left|\mcIRK\cup\mcID\right|\leq \left|\mcIRK\right|+\left|\mcID\right|\leq2\left|\mcIRK\right|$ and
\begin{equation}\label{eq:CardIRK}
|\mcIRK| \leq (^{n-1}_{K-1})^4 \leq 
C_1\n^{4\Kmax},
\end{equation}
for some positive constant $C_1$.

\subsection*{Proof of Lemma \ref{Lemme:InegaliteConcentrationZn}}

Using (\ref{eq:comptage_nkl}), we obtain
$$
\sum_{(\ii,\jj)\in\Gzero}(\mbE\left[\Yij\right]-\muzs)
=\sum_{\el=0}^{\K^{\vrai}}\sum_{(\ii,\jj)\in\Gzero\cap\Dsl}(\mbE\left[\Yij\right]-\muzs)=\sum_{\el=0}^{\K^{\vrai}}\nzl(\muls-\muzs)
=\sum_{\el=1}^{\K^{\vrai}}\nzl(\muls-\muzs).
$$
By (\ref{eq:Z}), we thus obtain that
$$
\Zn(\bt)=\frac{4}{\n(\n+1)}\frac{1}{|\Gun|}\left(\sum_{(\ii,\jj)\in\Gun}{\epsij}\right) \left[\left(\sum_{(\ii,\jj)\in\Gszero}{\epsij}-\sum_{(\ii,\jj)\in\Gzero}{\epsij}\right)-\sum_{\el=1}^{\K^{\vrai}}\nzl(\muls-\muzs)\right].
$$
This gives the following upper bound for $\Zn$:
\begin{eqnarray*}
|\Zn(\bt)|
%
&\leq&\frac{8}{\n(\n+1)}\left[\frac{\left|\sum_{(\ii,\jj)\in\Gun}{\epsij}\right|}{|\Gun|}\right] \times\left[\underset{\A\in\mcIRK}{\max}\left|\sum_{(\ii,\jj)\in\A}{\epsij}\right|\right] + \frac{4}{\n(\n+1)}\left[\frac{\left|\sum_{(\ii,\jj)\in\Gun}{\epsij}\right|}{|\Gun|}\right] n_{0}\lambdasup,\\
\end{eqnarray*}
where $\mcIRK$ is defined in (\ref{eq:def_IRK}) and where we used:
\begin{equation*}
\sum_{(\ii,\jj)\in\Gszero}{\epsij}-\sum_{(\ii,\jj)\in\Gzero}{\epsij}=\sum_{(\ii,\jj)\in\Gszero\cap\compl{\Gzero}}{\epsij}
-\sum_{(\ii,\jj)\in\Gzero\cap\compl{\Gszero}}{\epsij},
\end{equation*}
$\compl{A}$ denoting the complement of the set $A$.
Thus,
\begin{eqnarray*}
\Prob\left(- \underset{\bt\in\AnK}{\min}\Zn(\bt)\geq\alpha\right)
&\leq& \probT{ \frac{8}{\n(\n+1)}\left[\frac{\left|\sum_{(\ii,\jj)\in\Gun}{\epsij}\right|}{|\Gun|}\right] \times\left[\underset{\A\in\mcIRK}{\max}\left|\sum_{(\ii,\jj)\in\A}{\epsij}\right|\right] \geq \alpha/2 }\\
&& + \probT{ \frac{4}{\n(\n+1)}\left[\frac{\left|\sum_{(\ii,\jj)\in\Gun}{\epsij}\right|}{|\Gun|}\right] n_{0}\bar{\lambda} \geq \alpha/2 }=:P_1+P_2.
\end{eqnarray*}
We shall now provide upper bounds for $P_1$ and $P_2$.

Let us now provide an upper bound for $P_1$.
Let $Z_1$ and $Z_2$ be two non negative random variables, then for all positive $\alpha$ and $\gamma$,
$$
\probT{Z_1Z_2 \geq \alpha } \leq \probT{Z_1 \geq \gamma} + \probT{Z_2 \geq \frac{\alpha}{\gamma}}.
$$
Applying this inequality to $P_1$ gives
\begin{eqnarray}\label{eq:P1:dec}
P_1 
&\leq& \probT{ \left[\frac{\left|\sum_{(\ii,\jj)\in\Gun}{\epsij}\right|}{|\Gun|}\right] \geq \gamma} + \probT{ \frac{8}{\n(\n+1)} \left[\underset{\A\in\mcIRK}{\max}\left|\sum_{(\ii,\jj)\in\A}{\epsij}\right|\right] \geq \frac{\alpha}{2\gamma} }.
\end{eqnarray}
By Lemma \ref{concentration.subgauss.mean}, we obtain
\begin{eqnarray}\label{eq:P11}
\probT{ \left[\frac{\left|\sum_{(\ii,\jj)\in\Gun}{\epsij}\right|}{|\Gun|}\right] \geq \gamma} \leq 2e^{-\frac{|\Gun|\gamma^2}{4\beta}}.
\end{eqnarray}
Moreover, by Lemma \ref{concentration.max.set}, we obtain
\begin{eqnarray}\label{eq:P12}
\probT{ \frac{8}{\n(\n+1)} \left[\underset{\A\in\mcIRK}{\max}\left|\sum_{(\ii,\jj)\in\A}{\epsij}\right|\right] \geq \frac{\alpha}{2\gamma} }
\leq 2|\mcIRK| e^{-\frac{\alpha^2\n(\n+1)}{512\gamma^2\betaa}}\leq 2C_1\n^{\Kmax} e^{-\frac{\alpha^2\n(\n+1)}{512\gamma^2\betaa}},
\end{eqnarray}
where we used that $|\mcIRK| \leq C_1\n^{4\Kmax}$ (see Equation~(\ref{eq:CardIRK})).


Let us now provide an upper bound for $P_2$.
Using Lemma \ref{concentration.subgauss.mean} we obtain the following upper bound:
\begin{eqnarray}\label{eq:P2}
P_2 
\leq 2e^{-\frac{|\Gun|\alpha^2n^2}{64\bar{\lambda}^2\beta} },
\end{eqnarray}
where we used that $n_0\leq n(n+1)/2$.
The proof of Lemma \ref{Lemme:InegaliteConcentrationZn} thus follows from (\ref{eq:P1:dec}), (\ref{eq:P11}), (\ref{eq:P12}) and (\ref{eq:P2}).

\subsection*{Technical lemmas}
\begin{lemma}\label{concentration.subgauss.mean}
Let $A$ be a subset of $\{1,\dots n\}^2$. Assume that (A\ref{hyp:eps}) holds then for all positive $\alpha$,
\begin{equation*}
\probT{\left|\sum_{(i,j)\in A} \varepsilon_{i,j} \right| \geq \alpha} \leq 2e^{-\frac{\alpha^2}{4\beta|A|}}.
\end{equation*}
\end{lemma}

\begin{proof}[Proof of Lemma \ref{concentration.subgauss.mean}]
By the Markov inequality and (A\ref{hyp:eps}), we get for all positive $\eta$, that
$$
\probT{\sum_{(i,j)\in A} \varepsilon_{i,j} \geq \alpha}\leq\exp(-\eta\alpha+\beta |A|\eta^2).
$$
Taking $\eta=\alpha/(2\beta |A|)$ gives
$$
\probT{\sum_{(i,j)\in A} \varepsilon_{i,j} \geq \alpha}\leq e^{-\frac{\alpha^2}{4\beta|A|}}.
$$
This concludes the proof since the same bound holds when $\varepsilon_{i,j}$ is replaced by $-\varepsilon_{i,j}$.
\end{proof}

\begin{lemma}\label{concentration.max.set}
Assume that (A\ref{hyp:eps}) holds then, for all $\alpha>0$,
\begin{eqnarray*}
\Prob\left(\underset{\A\in\mcI}{\max}\left|\sum_{(i,j)\in\A}{\epsij}\right|\geq\alpha\right)\leq 2|\mcI|e^{-\frac{\alpha^2}{2\n(\n+1)\betaa}},
\end{eqnarray*}
where $\A$ is any subset of $\{(i,j): 1 \leq i \leq j \leq n\}$ and $\mcI$ is a collection of any such subsets $\A$.
\end{lemma}
\begin{proof}[Proof of Lemma \ref{concentration.max.set}]
By Lemma \ref{concentration.subgauss.mean},
\begin{eqnarray*}
\Prob\left(\underset{\A\in\mcI}{\max}\left|\sum_{(i,j)\in\A}{\epsij}\right|\geq\alpha\right)
&\leq& \sum_{\A\in\mcI}\Prob\left(\left|\sum_{(i,j)\in\A}{\epsij}\right|\geq\alpha\right)
\leq 2\sum_{\A\in\mcI} e^{-\frac{\alpha^2}{4|\A|\beta}} 
\leq 2|\mcI|e^{-\frac{\alpha^2}{2\n(\n+1)\betaa}},
\end{eqnarray*}
since $|A|\leq n(n+1)/2$.
\end{proof}



\end{document}